\documentclass[12pt,reqno,a4paper]{amsart}
\usepackage{amsfonts}
\usepackage{amsthm}
\usepackage{amsmath}

\usepackage{amsrefs}
\usepackage{amscd}
\usepackage[latin2]{inputenc}
\usepackage{t1enc}
\usepackage{amssymb}
\usepackage[mathscr]{eucal}
\usepackage{indentfirst}
\usepackage{graphicx}
\usepackage{graphics}
\usepackage{pict2e}
\usepackage{epic}
\numberwithin{equation}{section}
\usepackage[margin=2.9cm]{geometry}
\usepackage{epstopdf} 
\usepackage[all,cmtip]{xy}
\usepackage{enumerate}
\usepackage{pdfpages}
\usepackage[colorlinks = true,
            linkcolor = blue,
            urlcolor  = magenta,
            citecolor = blue,
            anchorcolor = blue]{hyperref}
\usepackage[hyphenbreaks]{breakurl}

\usepackage{relsize} 
\usepackage[bbgreekl]{mathbbol} 
\usepackage{amsfonts} 
\DeclareSymbolFontAlphabet{\mathbb}{AMSb} 
\DeclareSymbolFontAlphabet{\mathbbl}{bbold}

\usepackage{tikz}
\usetikzlibrary{matrix,arrows,decorations.pathmorphing}

\theoremstyle{plain}
\newtheorem{Th}{Theorem}[section]
\newtheorem{Lemma}[Th]{Lemma}
\newtheorem{Cor}[Th]{Corollary}
\newtheorem{Prop}[Th]{Proposition}

\theoremstyle{definition}
\newtheorem{Def}[Th]{Definition}
\newtheorem{Con}[Th]{Construction}

\newtheorem{Rem}[Th]{Remark}

\newtheorem{?}[Th]{Problem}
\newtheorem{Ex}[Th]{Example}

\newcommand{\Hom}{{\rm{Hom}}}
\newcommand{\RHom}{{\rm{RHom}}}
\newcommand{\Spf}{{\rm{Spf}}}
\newcommand{\Ext}{{\rm{Ext}}}

\newcommand{\Spec}{{\rm{Spec}}}

\newcommand{\colim}{\mathop{\mathrm{colim}}}

\newcommand{\tn}[1]{\textnormal{#1}}
\newcommand{\tbf}[1]{\textbf{#1}}
\newcommand{\Q}{\mathbb{Q}}
\newcommand{\Z}{\mathbb{Z}}
\newcommand{\DM}{\tn{\tbf{DM}}}
\newcommand{\et}{\tn{\'{e}t}}
\newcommand{\nis}{\tn{Nis}}
\newcommand{\DMe}{\DM^{\tn{eff}, -}}
\newcommand{\eff}{\tn{eff}}
\newcommand{\DMeet}{\DMe_{\et}}
\newcommand{\DMeN}{\DMe_{\nis}}

\begin{document}

\setcounter{page}{1}

\title{On the weight zero motivic cohomology}
\author{Semen Molokov}
\address{Department of Mathematics, HSE University, Ulitsa Usacheva 6, Moscow 119048, Russia}
\email{sam-molokov1@yandex.ru, sam\_molokov@yahoo.com}
\author{Vadim Vologodsky}
\address{Department of Mathematics, University of Chicago,
5734 S. University Av., Chicago, IL, 60637,
USA; Department of Mathematics, HSE University, Ulitsa Usacheva 6, Moscow 119048, Russia
}
\email{vologod@gmail.com}
\maketitle

\setcounter{tocdepth}{1}

\begin{center}

\begin{minipage}[t]{0.8\linewidth}

Abstract. We prove that singular cohomology of the underlying space of Berkovich's analytification of a scheme $X$ locally of finite type over a trivially-valued field $k$ of characteristic $0$ is isomorphic to cdh-cohomology with integer coefficients which is also isomorphic to the weight zero motivic cohomology $H^*(X, \Z)$.  Using this isomorphism, we demonstrate  the vanishing of $\RHom_{Sh_{Nis}(cor_k)}(\underline{G},\Z)$, where $\underline{G}$ denotes the Nisnevich sheaf with transfers associated with a commutative algebraic group $G$ over $k$. For abelian $k$-varieties $A$ and $B$, we prove that $\RHom_{\mathcal{PS}h_{tr}}(\underline{A},\underline{B})$ is isomorphic to $\Hom_{\mathbf{Ab_k}}(A,B)$. 
\end{minipage}
\end{center}

\tableofcontents
\section{Introduction}\label{Sec1}
There is a long history of comparison results between cohomology in algebraic and analytic geometry. Let us begin with two fundamental examples.

Let $X$ be a projective variety over a field $k$. When $k=\mathbb{C}$, the set of complex-valued points $X(\mathbb{C})$ of $X$ can be given the structure of a complex analytic space, denoted by $X^{an}$. For any sheaf $\mathcal{F}$ on $X$, one can associate a sheaf $\mathcal{F}^{an}$ on the analytification $X^{an}$. Serre's GAGA theorem says that if $\mathcal{F}$ is a coherent sheaf on $X$, then the functorial homomorphisms $\epsilon_q\colon H^q(X, \mathcal{F})\to H^q(X^{an}, \mathcal{F}^{an})$ are isomorphisms for all $q$.

Another fundamental connection was established by Grothendieck. Let $X$ be a smooth algebraic variety over $k$. Algebraic de Rham cohomology $H^*_{dR}(X/k)$ are defined as hypercohomology of the de Rham complex
\[
\xymatrix{
\Omega^0_{X/k}\ar[r]^d & \Omega^1_{X/k}\ar[r]^d & \Omega^2_{X/k}\ar[r]^d & \ldots
}.
\]
When $k=\mathbb{C}$, Grothendieck showed that the analytification comparison map induces an isomorphism $H^*_{dR}(X/k)\simeq H^*_{dR}(X^{an})$ between algebraic and analytic de Rham cohomology. Moreover, by de Rham's theorem, the latter one is isomorphic to singular cohomology $H^*_{sing}(|X^{an}|, \mathbb{C})$. 

When $k=\Q_p$ or $\mathbb{C}_p$, the induced topology on $k$ is totally disconnected, which presents challenges in developing a suitable framework for $k$-analytic geometry.  One prominent  approach is due to Berkovich, who introduced $k$-analytic spaces (now known as Berkovich spaces). These spaces have desirable topological properties and can be defined even over trivially-valued fields. Furthermore, there is an analytification functor $X\mapsto X^{an}$ from schemes locally of finite type over $k$ to Berkovich spaces that preserves many properties of morphisms between schemes. Another sign that Berkovich spaces are useful is the existence of \'etale cohomology theory constructed in \cite{BerII}.  Throughout this paper, we usually endow the based field $k$ with a trivial norm.

The theory of motives provides a universal framework for unifying various similarly behaved cohomology theories of varieties over a field $k$. Working with Voevodsky motives requires the notion of the Nisnevich topology, which retains useful properties of both Zariski and \'etale topologies. Then, for any smooth $k$-variety $X$, one can associate a Nisnevich sheaf with transfers $\Z_{tr}[X]$ and a motive $M(X)\in \DMeN(k, \Z)$. To study singular varieties, one must extend covers in the Nisnevich topology by incorporating some form of resolution of singularities. For instance, the cdh-topology allows proper birational maps as coverings. Our first result gives a connection between cdh-cohomology $H^*_{cdh}(X, \Z)$ of a connected scheme $X$ locally of finite type over $k$ and usual singular cohomology $H^*(|X^{an}|, \Z)$ of the underlying space of the analytification:
\begin{Th}\label{Int1}
Let $X$ be a connected scheme locally of finite type over a field $k$ which is trivially-valued and allows resolution of singularities. Then there is a canonical quasi-isomorphism
\[
    C^\bullet(|X^{an}|, \Z)\to R\Gamma_{cdh}(X, \Z),
\]
where $C^\bullet(|X^{an}|, \Z)$ is the usual singular cochain complex computing singular cohomology of $|X^{an}|$ with $\Z$ coefficients.
\end{Th}
We give the proof of this theorem in section~\ref{Sec4} (Theorem~\ref{Comparison}). The proof is based on Voevodsky's results concerning presheaves on topologies generated by cd-structures and several results on the homotopy type of $|X^{an}|$ due to Berkovich and Thuillier. 

Let  $cor_k$ be the category whose objects are smooth $k$-varieties and whose morphisms are finite correspondences. The category of presheaves with transfers, denoted by $PSh(cor_k)$, consists of presheaves $F\colon cor_k^{op}\to \mathbf{Ab}$.  We denote the category of Nisnevich sheaves with transfers by $Sh_{Nis}(cor_k)$. For any commutative algebraic group $G$ over $k$, the presheaf $X\mapsto Mor_{Sch_k}(X, G)$ admits a natural transfer structure. We denote this presheaf by $\underline{G}$. When $k$ is a perfect field, Spie{\ss} and Szamuely proved that $\underline{G}$ is a Nisnevich sheaf with transfers, see \cite[Lemma 3.2]{SS}. Furthermore, if $G$ is a semi-abelian variety, Orgogozo's results \cite[Lemma 3.3.1]{Org} show that the sheaf $\underline{G}$ is also $\mathbb{A}^1$-homotopy invariant. We denote by $M_1(G)$ the image in $\DMeN(k, \Z)$ of the complex given by $\underline{G}$ concentrated in degree $0$.

Our next results reveal several new properties of $\underline{G}$ and $M_1(G)$. First, we prove the following theorem.
\begin{Th}
Let $G$ be a semi-abelian variety over $k$ of characteristic $0$. Then one has that 
\[
    \RHom_{\DMeN(k,\Z)}(M_1(G), \Z)\simeq 0.
\]
Moreover, if $G$ is only a commutative algebraic group over a field $k$ of characteristic~$0$ then $\RHom_{Sh_{Nis}(cor_k)}(\underline{G}, \Z)\simeq 0$.
\end{Th}

To prove this result, we develop some simplicial machinery using symmetric powers of algebraic varieties. For any commutative algebraic group $G$ over $k$, this machinery produces a special augmented simplicial commutative monoidal scheme over $k$ of the following form:
\[
    \xymatrix{
    \ldots \ar@<1ex>[r]\ar@<0ex>[r]\ar@<-1ex>[r] &S^\bullet(S^\bullet(G))\ar@<-.5ex>[r] \ar@<.5ex>[r] & S^\bullet(G)\ar[r] & G,
    }
\]
where $S^\bullet(G)$ denotes the infinite symmetric power of $G$.
We then show that this augmented simplicial commutative monoidal scheme is contractible when viewed as a simplicial scheme over $k$. By applying Quillen's results on homotopy theoretic group completion, we obtain the following proposition:
\begin{Prop}\label{int_acyclic}
Let $G$ be a commutative algebraic group over $k$. Then the above augmented simplicial $k$-scheme produces a resolution of $\underline{G}$ in $PSh(cor_k)$  (and in $Sh_{Nis}(cor_k)$) of the following form:
\[
    \xymatrix{
    \ldots\ar[r] &\Z_{tr}[S^\bullet(S^\bullet(G))]\ar[r] &\Z_{tr}[S^\bullet(G)]\ar[r]&\Z_{tr}[G]\ar[r]&\underline{G}.
    }
\]
\end{Prop}
To compute $\Ext_{Sh_{Nis}(cor_k)}^i(\underline{G}, \Z)$, we need to show that the resolution in Theorem~\ref{int_acyclic} is $\Hom_{Sh_{Nis}(cor_k)}(-, \Z)$-acyclic. This is done using Theorem~\ref{Int1}, as results by Berkovich and Thuillier (explained in section~\ref{Sec3}) show that for any smooth separated $k$-variety $X$ with an action of a finite group $G$, the analytification of the quotient $|X^{an}|/G$ is contractible.

However, since $S^\bullet(G)$ has infinitely many connected components, applying this resolution to compute $\Ext_{Sh_{Nis}(cor_k)}^i(\underline{G}, \Z)$ becomes challenging. To overcome this difficulty, we develop a similar approach using iterated reduced symmetric powers of pointed quasi-projective algebraic $k$-varieties. The
key difference with the previous approach is that the reduced symmetric power exists only as a pointed strict ind-scheme, not in the category of quasi-projective varieties. Using similar arguments, we obtain an augmented simplicial presheaf of pointed monoids on affine $k$-schemes of the form:
\[
    \xymatrix{
    \ldots \ar@<1ex>[r]\ar@<0ex>[r]\ar@<-1ex>[r] &S^\bullet_{red}(S^\bullet_{red}(G))\ar@<-.5ex>[r] \ar@<.5ex>[r] & S^\bullet_{red}(G)\ar[r] & \underline{G}
    }
\]
This augmented simplicial presheaf is contractible when considered as a simplicial presheaf of pointed sets. For a presheaf of monoids $\mathcal{F}$, we denote its group completion by $\mathcal{F}^+$. We can now state the following proposition, which provides a pointed version of Proposition~\ref{int_acyclic}:
\begin{Prop}\label{int_acyclic_pointed}
Let $G$ be a commutative algebraic group over $k$. Then the above augmented simplicial presheaf of pointed sets on affine $k$-schemes induces a resolution of $\underline{G}$ in $PSh(cor_k)$ (and in $Sh_{Nis}(cor_k)$) of the following form:
\[
    \xymatrix{
    \ldots\ar[r] & (S^\bullet_{red})^{\circ 3}(G)^+\ar[r] &(S^\bullet_{red})^{\circ 2}(G)^+\ar[r]&S^\bullet_{red}(G)^+\ar[r]&\underline{G}.
    }
\]
Moreover, this resolution is $\Hom_{Sh_{Nis}(cor_k)}(-, \Z)$-acyclic.
\end{Prop}
This resolution serves as our main computational tool for working with $\underline{G}$ in the category of presheaves with transfers (Nisnevich sheaves with transfers). 

Several important results are already known for $\underline{G}_\Q\simeq \underline{G}\otimes_\Z\Q$ and its associated motive $M_1(G)_\Q\in \DMeet(k, \Q)$. In particular, the following theorem \cite[Theorem~3.1.4]{AEWH} establishes a connection between the motive $M(G)\in \DMeet(k, \Q)$ and symmetric powers of motives $M_1(G)_\Q$:
\begin{Th}[Ancona, Enright-Ward, Huber]
Let $G$ be a semi-abelian variety over a perfect field $k$ which is an extension of an abelian variety of dimension $g$ by a torus of rank $r$. Then the motive $M_1(G)_\Q$ is odd of dimension $2g+r$ and the map
\[
    \phi_G\colon M(G)\to \bigoplus_{n=0}^{2g+r}S^n(M_1(G)_\Q)
\]
is an isomorphism of motives.
\end{Th}

Another significant result \cite[Proposition 3.3.3]{Org} is due Orgogozo. Let $\mathbf{sAb}$ denote the category of semi-abelian $k$-varieties and $\mathbf{Ab_k}$ the category of abelian $k$-varieties.
\begin{Prop}
The functor $M_1\colon G\mapsto M_1(G)_\Q$ is an exact functor from $\mathbf{sAb}$ to $\DMeet(k, \Q)$. It is isogeny invariant and induces an exact functor, which we also denote by $M_1(-)_\Q$, from $\mathbf{sAb}_{\Q}$ to $\DMeet(k, \Q)$. Moreover, we have
\[
    \Hom_{\mathbf{sAb}_{\Q}}(G, H)\simeq\Hom_{\DMeet(k, \Q)}(M_1(G)_\Q, M_1(H)_\Q).
\]
\end{Prop}
 
We obtain a similar result directly for $\underline{G}$ without tensoring with $\Q$. Let us denote by $\mathcal{PS}h_{tr}$ the category of presheaves on the category $Cor_k$. The objects of $Cor_k$ are separated schemes of finite type over $k$ such that any finite subset of any object $X\in Cor_k$ is contained in an affine open subset. The morphisms between $X$ and $Y$ are the elements of the group completion of the monoid $Mor_{Sch_{k}}(S^\bullet (X), S^\bullet (Y))$, and the composition of two morphisms is given by the usual composition of maps of schemes. We  denote by $Sh_{Nis}(Cor_k)$ the category of Nisnevich sheaves with transfers on $Cor_k$.

It turns out that the resolution of Proposition~\ref{int_acyclic_pointed} can be used for computations in $\mathcal{PS}h_{tr}$. Our final main result is the following theorem:

\begin{Th}\label{main_intro}
Let $A$ and $B$ be two abelian varieties over field $k$. Then the complex  $\RHom_{\mathcal{PS}h_{tr}}(\underline{A}, \underline{B})$ is quasi-isomorphic to $\Hom_{\mathbf{Ab_k}}(A, B)$.
In particular,
\[
    \Ext^i_{\mathcal{PS}h_{tr}}(\underline{A}, \underline{B})=0, \text{ for } i\geq 1.
\]
\end{Th}
We further extend this result to Nisnevich sheaves with transfers:
\begin{Cor}
Let $A$, $B$ be abelian varieties over $k$. Then $\RHom_{Sh_{Nis}(Cor_k)}(\underline{A}, \underline{B})$ is concentrated in degree $0$ and is quasi-isomorphic to $\Hom_{\mathbf{Ab_k}}(A, B)$.
\end{Cor}

This paper is organized as follows. In section~\ref{Sec2}, we review Voevodsky's theory of cd-structures.  In Section~\ref{Sec3}, we present essential results from Berkovich's work on the homotopy type of the analytification $X^{an}$ of an algebraic variety $X$ over a trivially-valued field $k$, following Thuillier's treatment in \cite{THU}. These two sections summarize known results and provide necessary background. In section~\ref{Sec4}, we prove the comparison isomorphism between $H^*(|X^{an}|,\Z)$ and $H^*_{cdh}(X, \Z)$. Section~\ref{Sec5} begins with a review of the general theory related to Voevodsky motives. We then apply simplicial machinery to obtain a resolution of $\underline{G}$ using iterated (reduced) symmetric power construction. To verify its $\Hom_{Sh_{Nis}(cor_k)}(-, \Z)$-acyclity, we use the comparison results from section~\ref{Sec4}. Then we use this resolution to show that $\RHom_{Sh_{Nis}(cor_k)}(\underline{G}, \Z)$ is trivial. Finally, for two abelian varieties $A$ and $B$ over $k$, we compute $\RHom_{\mathcal{PS}h_{tr}}(\underline{A}, \underline{B})$ which turns out to be isomorphic to $\Hom_{\mathbf{Ab_k}}(A,B)$.

\subsection{Acknowledgments} This work is part of the first author's Ph.D. thesis at HSE University, written under the supervision of the second. The first author would like to thank Amaury Thuillier for answering questions about his results on the homotopy type of Berkovich analytification presented in section~\ref{Sec3}. Research of V.V. was supported by
the University of Chicago, Princeton University, and the Simons Collaboration on Perfection
in Algebra, Geometry, and Topology.

\section{Cd-structures and cdh-topology}\label{Sec2}
In this section, we review the definitions of Nisnevich and cdh-topologies on the category $Sch_k$ using cd-structures and their associated topologies. This approach was developed by Voevodsky in \cite{VoevodI} and \cite{VoevodII}. For a concise introduction to these topics, we refer the reader to \cite[Section 3]{CHSW}.

\begin{Def}
    Let $\mathcal{C}$ be a category with an initial object. A cd-structure on $\mathcal{C}$ is a class $\mathcal{P}$ of commutative squares in $\mathcal{C}$ of the form
    \[
    \xymatrix{ W\ar[d]\ar[r] & V \ar[d]^p\\ U\ar[r]^e & X}
    \]
    that is closed under isomorphisms.
The squares in $\mathcal{P}$ are called distinguished. 
\end{Def}

\begin{Def}\label{topology}
    The topology $t_\mathcal{P}$ associated with a cd-structure $\mathcal{P}$ is the weakest Grothendieck topology such that for any distinguished square from $\mathcal{P}$, the sieve $(p, e)$ generated by the morphisms $\{p\colon V\to X, e\colon U\to X\}$ is a covering sieve, and the empty sieve is a covering sieve of the initial object $\emptyset$.
\end{Def}

In this paper, we mainly focus on Nisnevich and cdh-topologies which can be obtained as associated topologies of upper and combined cd-structures, respectively. 
\begin{Def}
    A cartesian square
    \[
    \xymatrix{ W\ar[d]\ar[r] & V \ar[d]\\ U\ar[r] & X}
    \]
    in $Sch_k$ is called an elementary Nisnevich square if $U\to X$ is an open embedding, $V\to X$ is \'etale and $(V-W)\to (X-U)$ is an isomorphism. The cd-structure where the distinguished squares are elementary Nisnevich squares is called the upper cd-structure or Nisnevich cd-structure. The associated topology is called the upper cd-topology. It is proved in \cite[Proposition 2.17]{VoevodII}  that it coincides with the Nisnevich topology.
\end{Def}

\begin{Def}
    A cartesian square
    \[
    \xymatrix{ \widetilde{Z}\ar[d]\ar[r] & \widetilde{X} \ar[d] \\ Z\ar[r] & X}
    \]
    in $Sch_k$ is called an abstract blow-up if $Z\to X$ is a closed immersion, $\widetilde{X}\to X$ is proper and $(\widetilde{X}-\widetilde{Z})_{red}\to (X-Z)_{red}$ is an isomorphism. The cd-structure where the distinguished squares are abstract blow-ups is called the lower cd-structure.
\end{Def}
\begin{Def}
    The union of the upper and lower cd-structures is called the combined cd-structure. The cdh-topology is the topology associated with the combined cd-structure.
\end{Def}

\begin{Def}
A presheaf $F^\bullet$ of cochain complexes on $Sch_k$ satisfies the Mayer-Vietoris property for a class of cartesian squares $\mathcal{P}$ if for any square in that class, the resulting square after applying $F^\bullet$ is a homotopy pullback. This is equivalent to saying that for any square 
\[
    \xymatrix{ W\ar[d]\ar[r] & V \ar[d]\\ U\ar[r] & X}
\]
in $\mathcal{P}$, the total complex of $(F^\bullet(X)\to F^\bullet(U)\bigoplus F^\bullet(V)\to F^\bullet(W))$ is acyclic. In particular,  any such square in $\mathcal{P}$ induces a long exact Mayer-Vietoris sequence in cohomology:
\[
    \ldots\to H^n(X, F^\bullet)\to
    H^n(U, F^\bullet)\oplus H^n(V, F^\bullet)\to
    H^n(W,F^\bullet)\to H^{n+1}(X, F^\bullet)\to\ldots.
\]
\end{Def}
We say that $F^\bullet$ satisfies Nisnevich descent for $Sch_k$ if $F^\bullet$ satisfies the MV-property for elementary Nisnevich squares, and $F^\bullet$ satisfies cdh-descent if $F^\bullet$ satisfies the Mayer-Vietoris property for all elementary Nisnevich and abstract blow-up squares. We note that Voevodsky used  the term "flasque" for presheaves with MV-property in \cite{VoevodI}.

\begin{Def}
Let $\mathcal{C}$ be a category with cd-structure $\mathcal{P}$. A presheaf $F^\bullet$ of cochain complexes on $\mathcal{C}$ is called a hypersheaf for the topology generated by $\mathcal{P}$ if the natural map $F^\bullet(U) \to R\Gamma(U, F^\bullet)$ is a quasi-isomorphism for each object $U\in\mathcal{C}$. The term quasifibrant is also used for such a presheaf.
\end{Def}

\begin{Th}\label{MV_quasifibrant}
    Let $Sch_k$ be the category of schemes of finite type over $k$, and let $\mathcal{P}$ be either upper or combined cd-structure on $Sch_k$. Then a complex of presheaves $F^\bullet$ is a hypersheaf if and only if it satisfies the Mayer-Vietoris property.
\end{Th}
\begin{proof}
   The proof is the combination of Theorem 3.4, Theorem 3.7 and Example 3.2 in  \cite{CHSW} and relies on results obtained by Voevodsky in \cite{VoevodI} and \cite{VoevodII}.
\end{proof}
\begin{Rem}
     Theorem~\ref{MV_quasifibrant} says that $F^\bullet$ satisfies Nisnevich/cdh-descent if and only if is a hypersheaf for the corresponding topology. This theorem will be crucial in section~\ref{Sec4}.
\end{Rem}

\section{Berkovich analytification}\label{Sec3}

In this section, we present an introduction to the results of Berkovich and Thuillier on the homotopy type of the analytification of a scheme $X$ over field $k$ equipped with a trivial valuation. We first recall the definitions of two types of $k$-analytic varieties that can be associated with a $k$-scheme $X$. Then, we review some necessary results on the homotopy type of the analytification of a toric $k$-variety. In the final subsection, we extend these results to more general $k$-schemes. Our main reference is \cite{THU}.

\subsection{Reminder on different types of analytifications}

\begin{Def}
    Let $X=\mathcal{M}(A)$ be a non-empty $k$-affinoid space. The union of all closed non-empty subspaces $\Gamma$ of $X$ such that $\max_X|f|=\max_{\Gamma}|f|$ for all elements $f\in A$ admits a minimal element denoted by $\Gamma(X)$. This is called Shilov boundary of $X$. This set is finite and has the following important properties:
    \begin{itemize}
        \item When $A$ is strictly $k$-affinoid, $\Gamma(X)$ is the inverse image under the reduction map $r_A\colon \mathcal{M}(A)\to \Spf(A^\circ)$ of the set of generic points of the special fiber of the formal scheme $\Spf(A^\circ)$. 
        \item For any non-archimedean extension $k'/k$, the Shilov boundary $\Gamma(X)$ is the image of $\Gamma(X_{k'})$ under the canonical projection of $\mathcal{M}(A\widehat{\otimes}_k k')$ on $X$.
    \end{itemize}
\end{Def}
Let $X=\mathcal{M}(A)$ be a $k$-affinoid space. There is a useful morphism of locally ringed $k$-spaces $\rho_A\colon\mathcal{M}(A)\to\Spec(A)$ that sends a semi-norm $x\in\mathcal{M}(A)$ to its kernel $\{f \in A| |f(x)|=0\}$ at the level of underlying spaces. It induces identity on $A$ (both sides of $\rho_A$ has $A$ as global sections of their corresponding structure sheaf). 
\begin{Def}
    Let $X$ be a scheme locally of finite type over $k$, and let $\Phi_X$ be the functor from the category of $k$-analytic spaces to the category of sets that assigns to every $k$-analytic space $\mathcal{Y}$ the set of morphisms of locally ringed $k$-spaces $\Hom_k(\mathcal{Y}, X)$. This functor is representable by a $k$-analytic space $X^{an}$ and a morphism of locally ringed $k$-spaces $\rho_X\colon X^{an}\to X$. 
\end{Def}

Another construction uses the fact that $k$ is equipped with a trivial valuation. This allows us to consider $X$ as a formal scheme over the ring $k^\circ =k$. We now recall the definition of the functor commonly called the "generic fiber".

Let $\mathfrak{X}$ be a formal scheme locally of finite type over $k$. We define the functor $\Phi_\mathfrak{X}$ from the category of $k$-analytic spaces to the category of sets as follows: to any $k$-analytic space $Y$, we associate the set of morphisms of $k$-ringed sites $\phi\colon Y_G\to \mathfrak{X}$ satisfying two conditions:
\begin{itemize}
    \item For any open affine $U\subset \mathfrak{X}$ and any affinoid domain $V$ of $Y$ with $\phi(V)\subset U$, the homomorphism $\phi^\# \colon\mathcal{O}_\mathfrak{X}(U)\to\mathcal{O}_{Y_G}(V)$ is continuous and bounded, i.e. $\|\phi^\#(a)|\leq 1$ for any $a\in\mathcal{O}_\mathfrak{X}(U)$.
    \item For any point $y\in Y$, the homomorphism
    \[
        \xymatrix{ \mathcal{O}_{\mathfrak{X},\phi(y)}\ar[r]^{\phi^\#_y} & \mathcal{O}_{Y,y}\ar[r] & \mathcal{H}(y)}
    \]    
    has image in $\mathcal{H}(y)^\circ$. Moreover, the induced homomorphism $\mathcal{O}_{\mathfrak{X},\phi(y)}\to \mathcal{H}(y)^\circ$ is local.
\end{itemize}
Here $Y_G$ denotes the site induced by the Grothendieck topology defined using analytic domains of $Y$.
\begin{Prop}
    Let $\mathfrak{X}$ be a scheme locally of finite type over $k$. Then the functor $\Phi_\mathfrak{X}$ is representable by a $k$-analytic space $\mathfrak{X}^\beth$ and a bounded morphism of $k$-ringed sites $r_\mathfrak{X}\colon \mathfrak{X}^\beth_G\to \mathfrak{X}$.
\end{Prop}
\begin{proof}
        See \cite[Proposition/Definition 1.3]{THU}.
\end{proof}

\begin{Rem}
For a scheme $X$ locally of finite type over $k$, we will compare $X^{an}$ and $X^\beth$ in Proposition~\ref{Compare}. It is worth noting that for any non-archimedean extension $K$ of $k$, the set $X^{an}(K)$ of $K$-points is precisely $X(K)$, while $X^\beth(K)$ is the set $X(K^\circ)$ of $K^\circ$-points of $X$. For example, if $X=\mathbb{A}^n_k$, then the underlying set of $X^{an}$ consists of all multiplicative semi-norms on $k[T_1,\ldots, T_n]$ extending the absolute value on $k$, while the underlying set of $X^\beth$ consists of all multiplicative semi-norms on $k[T_1,\ldots, T_n]$ that are bounded by $1$.
\end{Rem}
To establish a connection between $X^\beth$ and $X^{an}$, we first need to extend the definition of the map $\rho_A\colon\mathcal{M}(A)\to \Spec(A)$ to $X^\beth$.
\begin{Lemma} For any scheme $X$ locally of finite type over $k$, there is a unique map of locally-ringed $k$-spaces $\rho_X\colon X^\beth\to X$ such that for any affine $U\subset X$, the restriction of $\rho_X$ to affinoid domain $U^\beth$ coincides with $\rho_U\colon U^\beth\to U$.
\begin{proof}
     See \cite[Proposition 1.6]{THU}.
\end{proof} 
\end{Lemma}
\begin{Prop}\label{Compare}
Let $X$ be a scheme locally of finite type over $k$. We denote by $i_X\colon X^\beth \to X^{an}$ the morphism of $k$-analytic spaces induced by the morphism of locally ringed $k$-spaces $\rho_X\colon X^\beth \to X$. Then the following hold:
\begin{enumerate}
    \item For every open affine $U \subset X$, the restriction of $i_X$ to the affinoid domain $U^\beth$ is an isomorphism onto the affinoid domain of $X^{an}$.
    \item The morphism $i_X$ induces an isomorphism of $X^\beth$ onto the analytic domain of $X^{an}$ (resp. onto $X^{an})$ if and only if $X$ is separated (resp. proper).
    \item For any morphism $f\colon X\to Y$ of schemes locally of finite type over $k$, the diagram 
        \[
        \xymatrix{ X^\beth\ar[d]_{f^\beth}\ar[r]^{i_X} & X^{an} \ar[d]^{f^{an}} \\ Y^\beth\ar[r]_{i_Y} & Y^{an}}
        \]
    is commutative.
\end{enumerate}
\end{Prop}
\begin{proof}
     See \cite[Proposition 1.10]{THU}.
\end{proof}
\subsection{The case of a toric $k$-variety}
Let $T$ be a $k$-torus and $X$ be a toric variety over $k$, that is, an integral $k$-scheme endowed with an action of $T$ such that there is an open orbit $T$-isomorphic to $T$. We denote by $X_0$ the open orbit of $T$ in $X$, by $M=\Hom_{k-Groups}(T, \mathbb{G}_m)$ the group of characters of $T$, and for any character $m\colon T\to \mathbb{G}_m=\Spec (k[S, S^{-1}])$ we define $\chi_m=m^\#(S)$. The action of $T$ on $X$ induces an action of $k$-analytic group $T^\beth$ on $X^\beth$. 

\begin{Def}
For any point $x$ of $X^\beth$, we define $p(x)\in X^\beth$ to be the image of the unique Shilov point of the $\mathcal{H}(x)$-affinoid space $T^\beth \widehat{\otimes}_k\mathcal{H}(x)$ under the canonical map
 \[
        \xymatrix{ T^\beth\widehat{\otimes}_k\mathcal{H}(x)=T^\beth\times \mathcal{M}(\mathcal{H}(x))\ar[r]& T^\beth \times X^\beth \ar[r] &  X^\beth}.
\]

\end{Def}
\begin{Rem}
    \begin{enumerate}
        \item For any $K/k$ non-archimedean extension, the $K$-affinoid space $E=T^\beth \widehat{\otimes}_k K$ has a unique Shilov point.
        \item The map $p$ satisfies the identity $p^2=p$. We define the skeleton of $X$ to be $\Sigma(X)=p(X)=\{x\in X^\beth | p(x)=x\}$.
        \item The map $p$ is continuous, $\Sigma(X)$ is a closed subspace of $X^\beth$ and $p$ defines a retraction of $X^\beth$ onto $\Sigma(X)$.
    \end{enumerate}
\end{Rem}
\begin{proof}
    See \cite[Remark 2.2; Proposition 2.3; Proposition 2.9]{THU}.
\end{proof}

\begin{Rem}
    The closed subspace $\Sigma(X)$ of $X^\beth$ admits a description in terms of the fan of the toric $k$-variety $X$. If we define $\Sigma_0(X)=\Sigma(X)\cap \rho^{-1}(X_0)$, then it provides an intrinsic realization of the fan of the toric $k$-variety $X$ inside $X^\beth$. Then closed subspace $\Sigma(X)$ in the closure of $\Sigma_0(X)$ inside $X^\beth$. For more precise statements, see \cite[Section 2]{THU}.
\end{Rem}

\begin{Lemma}
    Let $t\in [0, 1]$ and $K/k$ be a non-archimedean extension. Then the $K$-affinoid space 
    \[
    G_K(t)=\{x\in T^\beth\widehat{\otimes}_k K\colon |\chi_m -1|\leq t, m\in M\} 
    \]
    is a $K$-analytic subgroup of $T^\beth\widehat{\otimes}_k K$ which has a unique Shilov point denoted by $g_K(t)$.
\end{Lemma}
\begin{proof}
    See \cite[Lemma 2.15]{THU}.
\end{proof}
Now we describe the homotopy which induces the strong deformation retraction of $X^\beth$ onto $\Sigma(X)$.
\begin{Def}
    We define the homotopy $H\colon [0, 1]\times X^\beth \to X^\beth$ as follows: for $t\in [0,1]$ and $x\in X^\beth$, let $H(t, x)$ be the image of the point $g_K(t)$ of $T^\beth\widehat{\otimes}_k\mathcal{H}(x)$ under the canonical map
     \[
        \xymatrix{ T^\beth\widehat{\otimes}_k\mathcal{H}(x)=T^\beth\times \mathcal{M}(\mathcal{H}(x))\ar[r]& T^\beth \times X^\beth \ar[r] &  X^\beth}.
    \]  
\end{Def}
\begin{Prop}
    \begin{enumerate}
        \item The map $H$ is continuous.
        \item For any $x\in X^\beth$, we have $H(0, x)=x$ and $H(1, x)=p(x)$.
        \item For any $t\in [0, 1]$, the restriction of $H(t, \cdot)$ to $\Sigma(X)$ is the canonical injection of $\Sigma(X)$ into $X^\beth$. Therefore, $H$ defines a strong deformation retraction of $X^\beth$ onto $\Sigma(X)$.
    \end{enumerate}
\end{Prop}
\begin{proof}
     See \cite[Theorem 2.17]{THU}.
\end{proof}

\subsection{General Case}

In this subsection, we generalize results for $k$-toric varieties to a special class of morphisms called toroidal embeddings. 
\begin{Def}
\begin{itemize}
   \item Let $X$ be a normal scheme over $k$. An open dense immersion $X_0\hookrightarrow X$ is called a simple toroidal embedding if any point $x\in X$ has an open neighborhood $U$ with an \'etale morphism $\gamma$ to a toric $k$-variety $Z$ such that $X_0\cap U$ is the inverse image of the open orbit $Z_0$ of $Z$.
    \item An open dense immersion $X_0\hookrightarrow X$ to a $k$-scheme $X$ is called a toroidal embedding if for every point $x\in X$ there exists a neighborhood $x\in U$ in $X$ and a surjective \'etale morphism $\delta\colon V\to U$ such that the open immersion $\delta^{-1}(X_0)\hookrightarrow V$ is a simple toroidal embedding.
\end{itemize}
\end{Def}
In this subsection, we will work only with simple toroidal embeddings $X_0\hookrightarrow X$.
\begin{Prop}
    For any simple toroidal embedding $X_0\hookrightarrow X$, there exists a unique continuous map $p_X\colon X^\beth \to X^\beth$ such that for all \'etale charts $\gamma\colon U\to Z$, we have $p_X(U^\beth)\subset U^\beth$ and the diagram
   \[
        \xymatrix{ U^\beth\ar[d]_{\gamma}\ar[r]^-{p_X} & U^\beth \ar[d]^\gamma \\ Z^\beth\ar[r]_-{p_Z} & Z^\beth}
    \]
    is commutative. Here, $p_Z$ is the endomorphism of $Z^\beth$ introduced in the previous subsection on toric varieties. 

    Moreover, as in the toric case, we have $p_X^2=p_X$.
\end{Prop}
\begin{proof}
    See \cite[Corollary 3.13]{THU}.
\end{proof}
\begin{Rem}
    \begin{enumerate}
        \item  We define the skeleton of $X$ to be $\Sigma(X)=p_X(X^\beth)$. This is a closed subspace of $X^\beth$, and $p_X$ defines a retraction of $X^\beth$ onto $\Sigma(X)$.
        \item Let $\Sigma_0(X)=\Sigma(X)\cap \rho^{-1}(X_0)$. Then $\Sigma_0(X)$ has a natural rational conical polyhedral structure, and $\Sigma(X)$ is its closure in $X^\beth$. For more precise statements, see  \cite[Subsection 3.1.3]{THU}.
        \item For our purposes, it suffices to note that if $X$ is irreducible and smooth, then $X\hookrightarrow X$ is a simple toroidal embedding and $\Sigma(X)=o$, where $o$ is the distinguished point corresponding to the trivial absolute valuation on $k(X)$.
    \end{enumerate}
\end{Rem}

\begin{Def}
    Let $G$ be an algebraic group over $k$. We denote by $\widehat{G}_1$ the formal completion of $G$ at the unit point $1$.
\end{Def}

The main technical tool in constructing a homotopy similar to that of the previous subsection is the following result:

\begin{Lemma}\label{formal_extension}
    Let $G$ be a $k$-group with an action $m\colon G\times Z\to Z$ on a $k$-scheme $Z$. For any \'etale morphism $\gamma \colon U\to Z$, there exists a unique morphism of formal $k$-schemes $m_\gamma:\widehat{G}_1\times U \to U$ defining the action of $\widehat{G}_1$ on $U$ such that the  diagram 
    \[
    \xymatrix{\widehat{G}_1\times U \ar[r]^-{m_\gamma}\ar[d]_{id\times \gamma} & U\ar[d]^\gamma \\
    \widehat{G}_1\times Z \ar[r]_-m & Z}
    \]
    commutes. The bottom arrow is obtained by precomposing the action map $m$ with the canonical morphism $i\times id_Z\colon\widehat{G}_1\times Z\to G\times Z$. 
\end{Lemma}
\begin{proof}
    See \cite[Lemma 3.19]{THU}.
\end{proof}

Now let $\gamma\colon U\to Z$ be an \'etale chart on a toric variety $Z$ and let $T$ be the $k$-torus acting on $Z$ with a group of characters $M$. If we denote by $\mathfrak{T}$ the formal completion $\widehat{T}_1$, then
\[
\mathfrak{T}^\beth =\{x\in T^\beth \colon |\chi(m)-1|(x)<1, m\in M\}
\]
is a $k$-analytic subgroup of $T^\beth$. It is the union of the subgroups $G(t)$ for $t\in [0, 1)$ defined above.

Using Lemma~\ref{formal_extension}, we obtain the morphism of formal $k$-schemes $m_\gamma\colon \mathfrak{T}\times U\to U$. This induces the map between $k$-analytic spaces $m_\gamma\colon \mathfrak{T}^\beth\times U^\beth \to U^\beth$ which gives an action of the $k$-analytic group $\mathfrak{T}^\beth$ on $U^\beth$.

As usual, any point $x\in U^\beth$ defines a tautological $\mathcal{H}(x)$-point $\underline{x}\colon \mathcal{M}(\mathcal{H}(x))\to U^\beth$. This allows us to define a canonical morphism
\[
m_{\gamma, x}\colon \mathfrak{T}^\beth_{\mathcal{H}(x)}=\mathfrak{T}^\beth\times\mathcal{M}(\mathcal{H}(x))\to U^\beth
\]
by composing the homomorphism $id_{\mathfrak{T}^\beth}\times\underline{x}\colon \mathfrak{T}^\beth\times \mathcal{M}(\mathcal{H}(x))\to \mathfrak{T}^\beth\times U^\beth$ with $m_\gamma$.

This construction leads to the following definition:
\begin{Def}\label{local_def}
    Let $U\to Z$ be any \'etale chart on a toric variety $Z$. We define a map $H_\gamma\colon [0, 1]\times U^\beth \to U^\beth$ as follows:
    \begin{itemize}
        \item For any $t\in [0, 1)$, let $H_\gamma(t, x)$ be the image of the Shilov point $g_{\mathcal{H}(x)}(t)$ of the $\mathcal{H}(x)$-affinoid space $G_{\mathcal{H}(x)}(t)$ via  $m_{\gamma, x}$.
        \item Let $H_\gamma(1, x)=p_X(x)$.
    \end{itemize}
\end{Def}

\begin{Prop}\label{local}
For any \'etale chart $\gamma: U\to Z$ over $X$, the map 
\[
H_\gamma\colon [0, 1]\times U^\beth \to U^\beth 
\]
satisfies the following properties:
\begin{enumerate}
    \item The map $H_\gamma$ is continuous.
    \item For any point $u\in U^\beth$, we have $H_\gamma(0, u)=u$ and $H_\gamma(1, u)=p_U(u)$.
    \item  For all $t\in[0, 1]$, the map $H_\gamma(t, \cdot)$ restricts to the identity on $\Sigma(U)$.
    \item For any two \'etale charts $\gamma \colon U\to Z$ and $\gamma'\colon U\to Z'$, we have that $H_\gamma=H_{\gamma'}$.
\end{enumerate}
\end{Prop}
\begin{proof}
     See \cite[Proposition 3.23 and Proposition 3.25]{THU}.
\end{proof}
Now we can state the main theorem characterizing the homotopy type of $X^\beth$.
\begin{Th}\label{MainHomotopy}
    There exist a unique map 
    \[
    H\colon [0, 1] \times X^\beth \to X^\beth
    \]
    such that for any \'etale chart $\gamma\colon U\to Z$, the map $H$ sends $[0, 1]\times U^\beth$ to $U^\beth$ and its restriction to $[0,1]\times U^\beth$ coincides with $H_\gamma$ from Definition~\ref{local_def}. Moreover, $H$ satisfies the following conditions:
    \begin{itemize}
        \item  For all $x\in X^\beth$, $H(0, x)=x$.
        \item  For all $x\in X^\beth$, $H(1, x)=p_X(x)$.
        \item For all $x\in \Sigma(X)$ and all $t\in [0, 1]$, $H(t, x)=x$, meaning that $H$ is a deformation retraction of $X^\beth$ onto $\Sigma(X)$.
    \end{itemize}
\end{Th}

\begin{proof}
    See \cite[Theorem 3.26]{THU}.
\end{proof}

\begin{Rem}\label{contracte}
An immediate consequence of Theorem~\ref{MainHomotopy} is that the analytification a scheme, which is isomorphic to an open subscheme of a proper smooth scheme, is contractible. In particular, when $char(k)=0$ and $X$ is a smooth separated scheme over $k$, then by Hironaka's result, $X^{an}$ is contractible. 
\end{Rem}
\begin{Lemma}\label{commutativity}
    Let $U_0\hookrightarrow U$ be a simple toroidal embedding and let $f\colon V\to U$ be an \'etale morphism. Then the open immersion $f^{-1}(U_0)\hookrightarrow V$ is a simple toroidal embedding, $\Sigma(V)=f^{-1}(\Sigma(U))$, and the natural diagram 
    \[
        \xymatrix{ [0, 1]\times V^\beth\ar[d]_{id \times f}\ar[r]^-{H_V} & V^\beth \ar[d]^f \\ [0, 1]\times U^\beth\ar[r]_-{H_U} & U^\beth}
    \]
    is commutative.
\end{Lemma}

\begin{proof}
     See \cite[Lemma 3.28]{THU}.
\end{proof}

We now give the following application of Lemma~\ref{commutativity}.
\begin{Lemma}\label{homotopy_quotient}
    Let $X$ be a smooth irreducible variety over $k$ with an action of a finite group $G$. Then the homotopy $H$ constructed in the Theorem~\ref{MainHomotopy} is $G$-equivariant. 
\end{Lemma}
\begin{proof}
    This follows from Lemma~\ref{commutativity}. Indeed, for any element $g\in G$, the map defined by it is \'etale, and since $X$ is smooth, the identity map $id\colon X\to X$ is a simple toroidal embedding. Therefore, the diagram
    \[
        \xymatrix{ [0, 1]\times X^\beth\ar[d]_{id \times g}\ar[r]^-{H} & X^\beth \ar[d]^g \\ [0, 1]\times X^\beth\ar[r]_-{H} & X^\beth}
    \]
    is commutative, which proves the assertion.
\end{proof}

\begin{Cor}\label{quotient}
If $X$ is a smooth irreducible variety over $k$ with an action of a finite group $G$, then the analytification of the quotient $ (X/G)^{an}\simeq X^{an}/G$ is contractible.
\end{Cor}
\begin{proof}
    This follows immediately from Remark~\ref{contracte} and Lemma~\ref{homotopy_quotient}
\end{proof}
\section{Singular Cohomology of analytification}\label{Sec4}
Let $X$ be a scheme over $k$. We endow $k$ with the trivial valuation and consider its analytification $X^{an}$ in the sense of Berkovich. Let $C^\bullet(|X^{an}|, \Z)$ be the singular cochain complex of the analytification of $X$ that computes its singular cohomology. Let us denote by $\mathcal{F}^\bullet$ the presheaf on $Sch_k$ which sends $X$ to the complex $C^\bullet(|X^{an}|, \Z)$. We begin with the following lemma:
\begin{Lemma}\label{Nisn-descent}
The presheaf $\mathcal{F}^\bullet$ corresponding to the singular cochain complex of analytification satisfies Nisnevich descent. In particular, every elementary Nisnevich distinguished square 
\[
\xymatrix{ W\ar[d]^q\ar[r]^j & V \ar[d]^p \\ U\ar[r]^i & X}
\]
induces a long exact sequence in singular cohomology of analytifications:
\[
\ldots\to H^n(|X^{an}|, \Z)\to
H^n(|U^{an}|, \Z)\oplus H^n(|V^{an}|, \Z)\to
H^n(|W^{an}|, \Z)\to\ldots.
\]
\end{Lemma}

\begin{proof}
We consider the pullback square of analytifications (this functor preserves pullbacks when working with trivially valued fields, see \cite[Theorem 3.5.1]{BerI}):
    \[
    \xymatrix{ W^{an}\ar[d]^{q^{an}}\ar[r]^{j^{an}} & V^{an} \ar[d]^{p^{an}} \\ U^{an}\ar[r]^{i^{an}} & X^{an}}
    \]

By \cite[Proposition 3.4.6 and discussion after Corollary 3.5.2 ]{BerI}, since $i$ is an open immersion, $i^{an}$ is also an open immersion. By \cite[Proposition 3.3.1]{BerII}, the map $p^{an}$ is \'etale, and thus by \cite[Proposition 3.2.7]{BerII}, we see that $p^{an}$ is an open map. Moreover, since $X-U\simeq V-W$, we have $|X^{an}|-|U^{an}|\simeq |V^{an}|-|W^{an}|$. Therefore, $|X^{an}|$ is covered by the open subsets $|U^{an}|$, $|V^{an}|$ with $|W^{an}|\simeq |U^{an}|\cap |V^{an}|$. The desired result follows from the standard Mayer-Vietoris sequence for an open covering of a topological space.
\end{proof}

\begin{Lemma}\label{cdh-descent}
The presheaf $\mathcal{F}^\bullet$ corresponding to the singular cochain complex of analytification satisfies MV-property for abstract blow-ups. In particular, every abstract blow-up square
\[
\xymatrix{ \widetilde{Z}\ar[d]^q\ar[r]^j & \widetilde{X} \ar[d]^p \\ Z\ar[r]^i & X}
\]
induces a long exact sequence in singular cohomology of analytifications:
\[
\ldots\to H^n(|X^{an}|, \Z)\to
H^n(|\widetilde{X}^{an}|, \Z)\oplus H^n(|Z^{an}|, \Z)\to
H^n(|\widetilde{Z}^{an}|, \Z)\to\ldots .
\]
\end{Lemma}

\begin{proof}
We consider the pullback square of analytifications:
    \[
    \xymatrix{ \widetilde{Z}^{an}\ar[d]^{q^{an}}\ar[r]^{j^{an}} & \widetilde{X}^{an} \ar[d]^{p^{an}} \\ Z^{an}\ar[r]^{i^{an}} & X^{an}}
    \]
By \cite[Corollary 3.5.2 and Proposition 3.4.7]{BerI}, since $i$ is a closed immersion, $i^{an}$ is also a closed immersion. Additionally, by \cite[Corollary 3.5.2 and Proposition 3.4.7]{BerI} the map $p^{an}$ is proper, and thus by \cite[Proposition 3.3.6]{BerI} we conclude that $p^{an}(\widetilde{X}^{an})$ is a closed $k$-analytic set in $X$. Similar to the previous lemma, $|X^{an}|$ is covered by the closed subsets $|\widetilde{X}^{an}|$, $|Z^{an}|$ with $|\widetilde{Z}^{an}|\simeq |\widetilde{X}^{an}|\cap |Z^{an}|$. The desired result follows from the standard Mayer-Vietoris sequence for a closed covering of a topological space \cite[II.5.5]{Iverson}.
\end{proof}

We prove that the presheaf of cochain complexes $\mathcal{F}^\bullet(-)=C^\bullet(|-^{an}|, \Z)$ satisfies MV-property for elementary Nisnevich and abstract blow up squares. By Theorem~\ref{MV_quasifibrant}, $\mathcal{F}^\bullet$ is a hypersheaf, yielding a natural isomorphism 
\[
\mathcal{F}^\bullet(U)\to R\Gamma(U, \mathcal{F}^\bullet)\simeq R\Gamma(U, \mathcal{F}_{cdh}^\bullet),
\]
where $\mathcal{F}^\bullet_{cdh}$ denotes the cdh-sheafification of $\mathcal{F}^\bullet$.

Our next objective  is to compute $\mathcal{F}^\bullet_{cdh}$. Since we assume that $k$ admits resolution of singularities, then not only does every $X$ in $Sch_k$ have an abstract blow-up $X'\to X$ with $X'$ smooth, but every proper birational cover $X'\to X$ has a refinement $X''\to X$ with $X''$ being smooth. This refinement is obtained as a composite of blow-ups along smooth centers. Consequently, every cdh-sheaf on $Sch_k$ is determined by its restriction to $Sm_k$.

The constant presheaf $\Z$ is a homotopy invariant presheaf with transfers. By \cite[Theorem 13.1]{MVW} the Nisnevich sheafification of a presheaf with transfers also admits transfers. Assuming resolution of singularities, we have the following sequence of isomorphisms for a connected and smooth $k$-scheme $X$:
\[
H^n_{cdh}(X, \Z)\simeq H^n_{Nis}(X, \Z)\simeq H^n_{Zar}(X, \Z).
\]

For the first isomorphism, see \cite[Corollary 5.12.3]{SV} or \cite[Proposition 13.27]{MVW}. For the second isomorphism, see \cite[Proposition 13.9]{MVW} or \cite[Theorem 4.5]{Beil_Vologod}. Moreover, since the sheaf $\Z$ is flasque, its Zariski cohomology $H^n_{Zar}(X, \Z)$ vanish for $n\geq 1$ and are isomorphic to $\Z$ for $n=0$.

By Remark~\ref{contracte}, we know that $|X^{an}|$ is contractible. Therefore, for a smooth scheme $X$ over $k$, we have $C^\bullet(|X^{an}|, \Z)\simeq \Z$. Since, a cdh-sheaf on $Sch_k$ is determined by its restriction to $Sm_k$, we obtain the following theorem

\begin{Th}\label{Comparison}
For any connected scheme $X$ locally of finite type over $k$, the canonical map 
\[
C^\bullet(|X^{an}|, \Z)\to R\Gamma(X_{cdh}, \Z)
\] 
is a quasi-isomorphism. In particular, we have the following isomorphism of cohomology groups:
\[
H^n(|X^{an}|, \Z)\simeq H^n_{cdh}(X, \Z).
\]
The same holds if $X$ is not connected. 
\end{Th}
Theorem~\ref{Comparison} provides a useful tool for computing cdh-cohomology of a scheme $X$ with $\Z$-coefficients. We will demonstrate an application of this result in the next section

\section{An application to motives}\label{Sec5}

Let $G$ be a commutative algebraic group over a perfect field $k$. Then the presheaf $\underline{G}$ of abelian groups represented by $G$ is a Nisnevich sheaf with transfers. Using Theorem~\ref{Comparison} and appropriate simplicial machinery, we construct a special resolution of $\underline{G}$ by iterating (reduced) symmetric powers of $G$.  Using this resolution, we prove that $\RHom_{Sh_{Nis}(cor_k)}(\underline{G}, \Z)$ is trivial.  Subsequently, we compute $\RHom_{\mathcal{PS}h_{tr}}(\underline{A}, \underline{B})$ in the category of presheaves with transfers, where $A$ and $B$ are two abelian varieties over $k$.

\subsection{A motive associated to an algebraic $k$-group} In this subsection, we explain why the presheaf represented by any commutative algebraic group $G$ over a perfect field $k$ is naturally a Nisnevich sheaf with transfers, which we denote by $\underline{G}$. Furthermore, when $G$ is a semi-abelian variety over $k$ then $\underline{G}$ is also $\mathbb{A}^1$-homotopy invariant.  We conclude with a brief review of some fundamental results in the theory of presheaves with transfers and Voevodsky motives.

Let $Sch=Sch_k$ be the category of all schemes over $k$ and $Sm=Sm_k$ be the subcategory of smooth $k$-varieties. 

\begin{Def}
    Let $X$ and $Y$ be schemes over $k$ such that $X$ is normal. We denote by $cor(X, Y)$ the free abelian group generated by irreducible closed subsets $W\subset X\times Y$ whose associated integral subscheme is finite and surjective over $X$. Such $W$ is called an elementary correspondence, and elements of $cor(X,Y)$ are referred to as finite correspondences. 
\end{Def}
A correspondence $\sum a_i\Gamma_i\in cor(X, Y)$ is said to be effective if all $a_i\geq 0$. Thus, $cor(X,Y)$ is the group completion of a commutative monoid of effective correspondences, which we denote by $cor(X, Y)^{\eff}$.

Let $X$, $Y$, $Z$ be schemes over $k$ such that $Y$ and $Z$ are normal. Then there is a composition map $cor(Y, X)\otimes cor(Z, Y)\to cor(Z, X)$. This composition is defined by $\gamma\otimes\gamma'\mapsto \gamma\gamma'$, where $\gamma\gamma'$ is the push-forward by the projection $X\times Y\times Z\to X\times Z$ of the cycle $(X\times\gamma')\cap(\gamma\times Z)$ on $X\times Y\times Z$. Since $Y$ is smooth, the cycle-theoretic intersection is well-defined. This composition is associative. 

A fundamental example of an elementary correspondence is the graph $\Gamma_f$ of a map $f\colon X\to Y$ between two smooth varieties over $k$ where $X$ is connected. When $X$ is not connected we take the sum of the components of $\Gamma_f$ to obtain an element in $cor(X, Y)$.

\begin{Def}
Let $cor_k$ be the category whose objects are smooth varieties over $k$,
and whose morphisms are finite correspondences. It is an additive category with $\emptyset$ as the zero object and disjoint union as the coproduct. It also admits a tensor product $X\otimes Y=X\times Y$, which makes $cor_k$ into a symmetrical monoidal category. 
\end{Def}

There is a faithful embedding $Sm_k\hookrightarrow cor_k$ which sends $X$ to $X$ and maps a morphism $X\to Y$ to its graph $\Gamma_f$ considered as a cycle in $X\times Y$. One can easily verify that $\Gamma_g\circ\Gamma_f$ equals $\Gamma_{g\circ f}$ for any $f\colon X\to Y$ and $g\colon Y\to Z$ in $Sm$. 

\begin{Def}
A presheaf with transfers is a contravariant additive functor from $cor_k$ to the category of abelian groups $\mathbf{Ab}$. We denote by $PSh(cor_k)$ the category whose objects are presheaves with transfers and whose morphisms are natural transformations. 
\end{Def}

The restriction forgetful functor from $PSh(cor_k)$ to presheaves of abelian groups on $Sm$ is exact and faithful, allowing us to view elements of $PSh(cor_k)$ as ordinary presheaves with some additional structure. For $X\in Sch$, we denote by $\Z_{tr}[X]$ the presheaf with transfers defined by $Y\mapsto cor(Y, X)$. This induces a functor from $Sch$ to $PSh(cor_k)$. When $X$ is smooth, $\Z_{tr}[X]$ is a projective object in $PSh(cor_k)$ by the Yoneda lemma. 

For a pointed scheme $(X, x)$ over $k$, we define $\Z_{tr}[X/x]$ as the cokernel of the map $x_*\colon\Z\to\Z_{tr}[X]$ which is induced by $x\colon \Spec(k)\to X$. The map $x_*$ splits the structure map $\Z_{tr}[X]\to \Z$, yielding a natural splitting $\Z_{tr}[X]\simeq\Z\oplus\Z_{tr}[X/x]$.

\begin{Def}
A presheaf with transfers $F$ is called $\mathbb{A}^1$-homotopy invariant if for every $X$, the map $p^*\colon F(X)\to F(X\times\mathbb{A}^1)$ is an isomorphism. 
\end{Def}

\begin{Con}
Let $\Delta^\bullet$ be the standard cosimplicial scheme with simplicies given by $\Delta^n=\{(t_0,\ldots, t_n)\in \mathbb{A}^{n+1} : \sum t_i =1\}$. For any presheaf of abelian groups $F$ on $Sm$, we define a simplicial presheaf by $U\mapsto F(U\times \Delta^\bullet)$, which we denote by $C_\bullet(F)$. When $F$ is a presheaf with transfers, $C_\bullet(F)$ is also a simplicial presheaf with transfers. Let $C_*(F)$ be the associated chain complex and $C^\Delta_*(F)$ be the associated normalized complex. For any presheaf $F$, the homology presheaves given by $U\mapsto H_nC_* F(U)$ are $\mathbb{A}^1$-homotopy invariant for all $n$. In particular, the morphism $\mathcal{F}\to H_0C_* F$ is the universal map from $\mathcal{F}$ to an $\mathbb{A}^1$-homotopy presheaf. 
\end{Con}

\begin{Def}
A presheaf with transfers $F$ is called a Nisnevich sheaf with transfers if its underlying presheaf is a Nisnevich sheaf on $Sm$.
\end{Def}

We note that by \cite[Lemma 6.2]{MVW}, a presheaf with transfers $\Z_{tr}[X]$ is an \'etale (Nisnevich) sheaf for any $k$-scheme $X$. We denote by $Sh_{Nis}(cor_k)$ the category of Nisnevich sheaves with transfers and write $\mathbf{D}^{-}$ for bounded above derived category $\mathbf{D}^{-}Sh_{Nis}(cor_k)$. Let $I_\mathbb{A}$ be the smallest thick subcategory of $\mathbf{D}^-$ containing every $\Z_{tr}[X\times\mathbb{A}^1]\to \Z_{tr}[X]$ and closed under direct sums. The quotient $\mathbf{D}^-/I_{\mathbb{A}}$ is the localization $\mathbf{D}^-[W^{-1}]$, where $W$ is the class of maps in $\mathbf{D}^-$ whose cone lies in $I_\mathbb{A}$. We define the triangulated category of Voevodsky motives as this localization and denote it by $\DMeN(k, \Z)$ or just $\DMeN$. For a smooth scheme $X$ over $k$, we denote by $M(X)$ the class of $\Z_{tr}[X]$ in $\DMeN(k, \Z)$.

The thick subcategory of $\DMeN(k, \Z)$ generated by such $M(X)$ is called the category of effective geometric motives and is denoted by $\mathbf{DM}^{\text{eff}}_{\text{gm}}(k, \Z)$. When $k$ admits resolution of singularities, it turns out that $\mathbf{DM}^{\text{eff}}_{\text{gm}}(k, \Z)$ contains $M(Y)$ for every $Y\in Sch_k$ and is generated by $M(X)$, where $X$ is smooth and projective. 

Recall that informally, finite correspondences are multi-valued maps. To make this more precise, we need the notion of symmetric power of a scheme $X$. 

\begin{Def}\label{symmetric}
    Let $X$ be a scheme of finite type over $k$ such that any finite subset of $X$ is contained in an affine open set (typically, $X$ is a quasi-projective variety over $k$). We define the $n^{th}$ symmetric power $S^n(X)$ of $X$ as the quotient of $X^n$ by the action of the symmetric group $S_n$. We denote by $S^\bullet(X)$ the disjoint union $\coprod_{n\geq 0}S^n(X)$. 
\end{Def}

Then $S^\bullet(X)$ is a universal commutative monoidal scheme over $k$ generated by $X$. The following result of Suslin and Voevodsky establishes  the precise connection between correspondences and symmetric powers.

\begin{Prop}[Suslin, Voevodsky]\label{SV}
Let $X$ and $Y$ be schemes over $k$ such that $X$ satisfies the conditions of Definition~\ref{symmetric} and $Y$ is normal. Then there is a natural morphism of monoids 
    \[
    g\colon cor_k(Y, X)^{\eff}\to Mor_{Sch}
    (Y, S^\bullet(X)).
    \]
When $char(k)$ is zero, $g$ is an isomorphism; otherwise it becomes an isomorphism after inverting $p$. Moreover, $g$ is compatible with composition when $Y$ is quasi-projective.
\end{Prop}

\begin{proof}
    For the proof see \cite[Theorem 6.8]{Sing} or \cite[Proposition 2.1.3]{Beil_Vologod}.
\end{proof}

To proceed further, we need to consider presheaves with transfer on a larger category of schemes. For this purpose, we denote by $\mathcal{PS}h_{tr}$ the category of presheaves of abelian groups on the category $Cor_k$. The objects of $Cor_k$ are separated schemes of finite type over $k$ with the property that any finite subset of any object $X\in Cor_k$ is contained in an affine open subset. The morphisms $Mor_{Cor_k}(X, Y)$ are given by the group completion of the monoid $Mor_{Sch_{k}}(S^\bullet (X), S^\bullet (Y))$. The composition of two morphisms is defined by the standard composition of maps between symmetric powers. We denote by $Sh_{Nis}(Cor_k)$ the category of Nisnevich sheaves with transfers on $Cor_k$.

Let $X$ and $Y$ be smooth schemes over a perfect field $k$, and let $G$ be a commutative algebraic group over $k$. We usually denote its identity element by $e_G=e$. For an elementary correspondence $Z\in cor(X, Y)$ associated with a closed integral subscheme $Z\subseteq X\times Y$, we can construct a morphism
\[
\beta(Z)=\beta_{X, Y, G}(Z)\colon Mor_{Sch}(Y, G)\to Mor_{Sch}(X, G).
\]
as follows. Let $d$ be the degree of the induced projection $Z\to X$. Using Proposition~\ref{SV}, we find a canonical map $\alpha_Z:X\to S^d(Y)$. We then define $\beta(Z)(g)$ as the composition:
\[
\xymatrix{
\beta(Z)(g)\colon X\ar[r]^-{\alpha_Z} & S^d(Y)\ar[r]^-{S^d(g)} & S^d(G)\ar[r]^-{\bar{m}^d} & G,}
\]
where $\bar{m}^d$ is induced by a $d$-fold multiplication map $G^d\to G$.
For a general correspondence, the construction extends by linearity.

This construction of $\beta$ enables us to view $\underline{G}=Mor_{Sch}(-, G)$ as a presheaf with transfers on smooth varieties over $k$. We recall the following result of Spie{\ss} and Szamuely.

\begin{Prop}
    For any commutative group scheme $G$ over a perfect field $k$, there exists a sheaf with transfers on big \'etale (Nisnevich) site given by $\underline{G}=Mor_{Sch}(-, G)$. Moreover, when $G$ is a semi-abelian variety, $\underline{G}$ is also $\mathbb{A}^1$-homotopy invariant.
\end{Prop}
\begin{proof}
    See \cite[Lemma 3.2]{SS} and \cite[Lemma 3.3.1]{Org}. 
\end{proof}
When $G$ is a semi-abelian variety over $k$, we denote by $M_1(G)$ the complex given by $\underline{G}$ concentrated in degree $0$ in $\DMeN(k, \Z)$.

\subsection{Simplicial machinery for symmetric power of a scheme}
In this subsection, we apply simplicial machinery to construct a resolution of $\underline{G}$ in $Psh(cor_k)$. Using Theorem~\ref{Comparison}, we then prove that this resolution is $\Hom_{Sh_{Nis}(cor_k)}(-, \Z)$-acyclic. 

Let us first review the relevant simplicial methods for constructing resolutions. 
\begin{Def}
An augmented simplicial object $\xymatrix{A_*\ar[r]^{\epsilon} & A_{-1}}$ is called right contractible if for all $n$ there are morphisms $f_n\colon A_n\to A_{n+1}$ (called extra degeneracies) such that $\epsilon f_{-1}=id$, $\partial_{n+1}f_n=id$ for $n\geq 0$, $\partial_0f_0=f_{-1}\epsilon$ and $\partial_if_n=f_{n-1}\partial_i$ for $0\leq i\leq n$.
    
Recall that the category $\Delta$ has a natural involution $^\vee$  which fixes every object $[n]\in\Delta$. This involution is defined on the morphisms by
    \[
    \partial_i^\vee=\partial_{n-i}\colon[n-1]\to n \text{ and } \sigma_i^\vee=\sigma_{n-i}\colon[n+1]\to [n].
    \]
For a simplicial object $A$, we define its front-to-back dual $A^\vee$ as the composition of $A$ with this involution. 
    
An augmented simplicial object $\xymatrix{A_*\ar[r]^{\epsilon} & A_{-1}}$ is called left contractible if its dual augmented simplicial object is right contractible. Explicitly, this means that there exist morphisms $f_n\colon A_n\to A_{n+1}$ for all $n$ satisfying $\epsilon f_{-1}=id$, $\partial_{0}f_n=id$ for $n\geq 0$, $\partial_{-1}f_0=f_{-1}\epsilon$ and $\partial_i f_n=f_{n-1}\partial_{i-1}$ for $0\leq i\leq n$.
\end{Def} 
\begin{Rem}
    If $\xymatrix{A_*\ar[r]^\epsilon & A_{-1}}$ is a contractible augmented simplicial object with values in an abelian category $\mathcal{A}$, then it is aspherical, meaning  $\pi_n(A_*)=0$ for $n\neq 0$ and $\epsilon\colon\pi_0(A_*)\simeq A_{-1}$. For further discussion, see \cite{Weibel}, end of Section 8.4. 
\end{Rem}
\begin{Con}\label{simp0}
Let $\mathcal{C}$ be the category of separated schemes over $k$ such that any finite subset of $X\in\mathcal{C}$ is contained in an affine open subset, and let $\mathcal{D}$ be the category of commutative monoidal separated schemes satisfying the same property. The symmetric power functor $S^\bullet$ maps $\mathcal{C}$ to $\mathcal{D}$ and is left adjoint to the forgetful functor $U\colon \mathcal{D}\to \mathcal{C}$. Thus, the composition $\perp = S^\bullet\circ U$ is a cotriple, and by \cite[8.6.4]{Weibel}, such a cotriple yields an augmented simplicial commutative monoidal scheme:
\[
    \xymatrix{
    \ldots \ar@<1ex>[r]\ar@<0ex>[r]\ar@<-1ex>[r] &\perp(\perp(D))\ar@<-.5ex>[r] \ar@<.5ex>[r] & \perp(D)\ar[r] & D
    }
 \]
for any $D\in\mathcal{D}$. In particular, for a commutative algebraic group $G$ over $k$, we obtain an augmented simplicial commutative monoidal $k$-scheme:
    \[
    \xymatrix{
    \ldots \ar@<1ex>[r]\ar@<0ex>[r]\ar@<-1ex>[r] &S^\bullet(S^\bullet(G))\ar@<-.5ex>[r] \ar@<.5ex>[r] & S^\bullet(G)\ar[r] & G.
    }
    \]
\end{Con}

\begin{Prop}\label{simp}
Let $G$ be a commutative algebraic group over $k$. The augmented simplicial commutative monoidal scheme
\[
\xymatrix{
\ldots \ar@<1ex>[r]\ar@<0ex>[r]\ar@<-1ex>[r] &S^\bullet(S^\bullet(G))\ar@<-.5ex>[r] \ar@<.5ex>[r] & S^\bullet(G)\ar[r] & G
}
\]
from Construction~\ref{simp0} is left contractible when viewed as simplicial $k$-scheme.
\end{Prop}
\begin{proof}
    This follows directly from  \cite[Proposition 8.6.10]{Weibel} with $\perp=S^\bullet\circ U$. 
\end{proof}

Our next goal is to use this construction to obtain a resolution of $\underline{G}$ in $PSh(cor_k)$. To pass from iterated symmetric powers of $G$ to correspondences, we require the following technical result due to Quillen:

\begin{Prop}\label{Quillen}
Let $M$ be a commutative simplicial monoid. If $M$ is group-like (i.e. $\pi_0(M)$ is a group), then the canonical map $M\to M^+$ to the degree-wise group completion of $M$ is a weak equivalence. 
\end{Prop}

\begin{proof}
By \cite[Lemma 2.3]{SSV}, the adjunction map $M\to \Omega(B(M)^\text{fib})$ to the homotopy theoretic group completion $\Omega(B(M)^\text{fib})$ is a weak equivalence. Furthermore, by \cite[Lemma 2.10 and Remark 2.13]{SSV} the natural map $\Omega(B(M)^\text{fib})\to M^+$ is also a weak equivalence. The result follows from these two facts.
\end{proof}

From this point forward, we assume that the base field $k$ has characteristic $0$.

\begin{Th}\label{acyclic}
Let $G$ be a commutative algebraic group over $k$. The augmented simplicial $k$-scheme of Proposition~\ref{simp} yields a resolution of $\underline{G}$ in $PSh(cor_k)$ of the form:
\[
    \xymatrix{
    \ldots\ar[r] &\Z_{tr}[S^\bullet(S^\bullet(G))]\ar[r] &\Z_{tr}[S^\bullet(G)]\ar[r]&\Z_{tr}[G]\ar[r]&\underline{G}.
    }
\]
Moreover, this resolution is $\Hom_{Sh_{Nis}(cor_k)}(-, \Z)$-acyclic.
\end{Th}
\begin{proof}
For any $U\in Sm$, we have the augmented commutative monoid
\[
\xymatrix{
\ldots \ar@<1ex>[r]\ar@<0ex>[r]\ar@<-1ex>[r] &Mor_{Sch}(U, S^\bullet(S^\bullet(G)))\ar@<-.5ex>[r] \ar@<.5ex>[r] & Mor_{Sch}(U, S^\bullet(G))\ar[r] & Mor_{Sch}(U, G)
}
\] 
which is contractible as simplicial set. We denote it by $S^\bullet(G)_*(U)\to G(U)$. Since is has extra-degeneracies, the geometric realization $|S^\bullet(G)_*(U)|$ is homotopy equivalent to $G(U)$ by \cite[Lemma 1.12]{ERW}. Consequently, $\pi_i(|S^\bullet(G)_*(U)|)=0$ for $i\geq 1$ and $\pi_0(|S^\bullet(G)_*(U)|)\simeq G(U)$. Thus, $S^\bullet(G)_*(U)\to G(U)$ is aspherical, meaning $\pi_i(S^\bullet(G)_*(U))=0$ for $i\geq 1$ and $\pi_0(S^\bullet(G)_*(U))\simeq G(U)$. In particular, the simplicial commutative monoid $S^\bullet(G)_*(U)$ is group-like. By Proposition~\ref{Quillen}, the natural map $S^\bullet(G)_*(U)\to S^\bullet(G)_*(U)^+$  is a weak equivalence. Therefore, the augmented simplicial abelian group  $S^\bullet(G)_*(U)^+\to G(U)$ is aspherical. Indeed, for $i\geq 1$, we have $\pi_i(S^\bullet(G)_*(U)^+)\simeq\pi_i(S^\bullet(G)_*(U))=0$, and for $i=0$, we conclude that $\pi_0(S^\bullet(G)_*(U)^+)\simeq\pi_0(S^\bullet(G)_*(U))\simeq G(U)$. By Proposition~\ref{SV}, we have $S^\bullet(G)_{n+1}(U)^+\simeq cor(U, (S^\bullet)^{\circ n}(G))$, where $(S^\bullet)^{\circ n}(G)$ denotes the symmetric power functor applied $n$ times to $G$. Thus, for any commutative algebraic $k$-group $G$ and any $U\in Sm$, we get an augmented aspherical simplicial abelian group 
\[
\xymatrix{
\ldots \ar@<1ex>[r]\ar@<0ex>[r]\ar@<-1ex>[r] &cor(U, S^\bullet(G))\ar@<-.5ex>[r] \ar@<.5ex>[r] &cor(U, G)\ar[r] & Mor_{Sch}(U, G).
}
\]
By Dold-Kan correspondence, the associated chain complex
\[
\xymatrix{
\ldots\ar[r]&cor(U, S^\bullet(G))\ar[r] \ar[r]&cor(U, G)\ar[r] & Mor_{Sch}(U, G)
}
\]
is exact, yielding an exact sequence of presheaves with transfers
\[
    \xymatrix{
    \ldots\ar[r] &\Z_{tr}[S^\bullet(S^\bullet(G))]\ar[r] &\Z_{tr}[S^\bullet(G)]\ar[r]&\Z_{tr}[G]\ar[r]&\underline{G}
    }
\]
which is also exact in $Sh_{Nis}(cor_k)$. 

Since each term $(S^\bullet)^{\circ n}(G)$ is a disjoint union of quotients of smooth $k$-varieties by finite groups, Corollary~\ref{quotient} and Theorem~\ref{Comparison} imply that $H^n_{cdh}((S^\bullet)^{\circ n}(G),\Z)\simeq 0$. Therefore, by \cite[Theorem 14.20]{MVW} we conclude that $\Ext^i_{Sh_{Nis}(cor_k)}((S^\bullet)^{\circ n}(G)],\Z)\simeq 0$ for all $n$. Consequently, each presheaf $\Z_{tr}[(S^\bullet)^{\circ n}(G)]$ is $\Hom_{Sh_{Nis}(cor_k)}(-, \Z)$-acyclic. 
\end{proof}

\begin{Rem}
The first part of the statement of Theorem~\ref{acyclic} remains valid if we replace $Psh(cor_k)$ by $\mathcal{PS}h_{tr}$. Indeed, the proof follows similar lines, requiring only that we take $U$ to be any quasi-projective $k$-variety. Also, this resolution is $\Hom_{\mathcal{PS}h_{tr}}(-,\Z)$-acyclic since all $\Z_{tr}[(S^\bullet)^{\circ n}(G)]$ are projective in $\mathcal{PS}h_{tr}$.
\end{Rem}

\subsection{Reduced symmetric power of a pointed scheme}
In this subsection, we define and study several properties of reduced symmetric powers of pointed quasi-projective varieties over $k$. While these properties parallel those of standard symmetric powers discussed in the previous subsection, there is a crucial distinction: reduced symmetric powers exist only as presheaves and do not belong to the category of quasi-projective varieties.

Let $\mathbf{Sets}$ denote the category of sets and $\mathbf{Sets}_*$ the category of pointed sets. We start by establishing some basic definitions and results.

\begin{Def}
    Let $\mathcal{C}$ and $\mathcal{D}$ be categories. A functor $F\colon \mathcal{C}\to \mathcal{D}$ is called final if for every object $d\in\mathcal{D}$, the comma category $(d/F)$ is non-empty and connected. Dually, $F$ is called initial if the opposite functor $F^{op}\colon \mathcal{F}^{op}\to \mathcal{D}^{op}$ is final, that is, if for any object $d\in\mathcal{D}$, the comma category $(F/d)$ is connected.
\end{Def}
\begin{Prop}
    Let $F\colon \mathcal{C}\to\mathcal{D}$ be a functor. Then the following are equivalent:
    \begin{enumerate}
        \item F is final.
        \item For all functors $G\colon\mathcal{D}\to\mathbf{Sets}$, the natural map between colimits 
        \[
        \colim G\circ F \to \colim G
        \]
        is a bijection.
        \item For all categories $\mathcal{E}$ and all functors $G\colon \mathcal{D}\to\mathcal{E}$, the natural map 
        \[
        \colim G\circ F\to \colim G
        \]
        is an isomorphism.
        \end{enumerate}
    Analogously, $F$ is initial if and only if pulling back diagrams along it does not change the limits of these diagrams.
\end{Prop}
\begin{proof}
    See \cite[Proposition 2.5.2]{KashShap}.
\end{proof}
\begin{Lemma}
    Every right adjoint functor is final.
\end{Lemma}
\begin{proof}
Let $(L\dashv R)\colon\mathcal{C}\to\mathcal{D}$ be adjoint functors. For any $d\in\mathcal{D}$, the adjunction unit $(L(d), d\to RL(d))$ lies in $(d/R)$. Moreover, any two objects
\[
    \xymatrix{
    R(a)& d\ar[r]^-{g}\ar[l]_-{f}& R(b)
    }
\]
are connected by a zig-zag going through the unit:
\[
    \xymatrix{
    & d\ar[d]\ar[ld]_f\ar[rd]^g \\
    R(a) & RL(d)\ar[r]_{R(g)}\ar[l]^{R(f)}& R(b).
    }
\]
\end{proof}
\begin{Ex}\label{colimits}
    Let $\mathcal{N}$ be the category whose objects are natural numbers and where the morphism $x\to y$ exists precisely when $x\geq y$.  Then the diagonal functor $\Delta\colon \mathcal{N}\to\mathcal{N}\times\ldots\times\mathcal{N}$ sending $i$ to $(i,\ldots, i)$ has the right adjoint given by $(n_1,\ldots, n_k)\mapsto \max(n_1,\ldots, n_k)$. Therefore, $\Delta$ is initial. Consequently, the opposite functor $\Delta^{op}\colon\mathcal{N}^{op}\to \mathcal{N}^{op}\times\ldots\times\mathcal{N}^{op}$ is final. 
\end{Ex} 

Let $\mathbf{Aff}$ denote the category of affine schemes over $k$, and let $\mathbf{Var_*}$ denote the category of pointed quasi-projective varieties over $k$. Let $Psh$ and $Psh_*$ be the categories of presheaves on $\mathbf{Aff}$ with values in $\mathbf{Sets}$ and $\mathbf{Sets}_*$, respectively.
\begin{Def}
    A strict ind-scheme is an object $X\in Psh$ that admits a presentation $X\simeq\colim_{i\in I} X_i$ as a filtered colimit of schemes, where all transition maps $X_i\to X_j$ are closed immersions. A pointed strict ind-scheme is defined analogously. We denote the corresponding categories by $\mathbf{IndSch}$ and $\mathbf{IndSch_*}$, respectively.
\end{Def}
\begin{Rem}
    We follow the definition of ind-schemes given in \cite{Richarz}. This definition agrees with the one given in \cite[7.11.1]{Beil_Drin}, except that we do not require the schemes in the presentation to be quasi-compact in order to make the category of schemes a full subcategory of ind-schemes. Our definition differs from \cite[0.3.4]{Zhu} in that we do not consider any Grothendieck topology.
\end{Rem}
To construct a pointed version of the resolution from  Theorem~\ref{acyclic}, we have to define the iterated reduced symmetric power of $(X, x_0)$.
\begin{Def}
    Let $\mathcal{G}$ be any object in $Psh_*$. We define the reduced symmetric power $S^\bullet_{red}\mathcal{G}$ as the left Kan extension of $S_{red}^\bullet\colon \mathbf{Var_*}\to Psh_*$ along the Yoneda embedding $\mathbf{Var_*}\to Psh_*$. Namely, we set 
    \[
    S^\bullet_{red}\mathcal{G}=\colim_{X\to \mathcal{G}}S^\bullet_{red}(X),
    \]
    where $X$ ranges over objects in $\mathbf{Var_*}$.
\end{Def}

\begin{Rem}
For a pointed quasi-projective $k$-variety $(X, x_0)\in\mathbf{Var_*}$,
we can compute the iterated reduced symmetric power as:
    \begin{align*}
        (S^\bullet_{red})^{\circ k}(X) &\simeq \colim_{n_1}\ldots\colim_{n_{k-1}} S^\bullet_{red}(S^{n_{k-1}}\ldots S^{n_1}(X) )\\
        &\simeq \colim_{(n_1,\ldots,n_k)} S^{n_k}\ldots S^{n_1}(X) \\
        &\simeq \colim_t S^t \ldots S^t(X),
    \end{align*}
where the last isomorphism follows from Example~\ref{colimits}.
We note that since all transition maps are closed immersions, it follows that $(S^\bullet_{red})^{\circ k}(X)$ is a pointed strict ind-scheme over $k$.
\end{Rem}
\begin{Lemma}\label{fpqc}
    Let $X$ be a strict ind-scheme over $k$. Then for each scheme $Y$ and each fpqc covering $(Y_j\to Y)$, the sequence of sets
    \[
    \xymatrix{
    \Hom_{\mathbf{IndSch}}(Y,X)\ar[r] & \prod_j\Hom_{\mathbf{IndSch}}(Y_j, X)\ar@<0.5ex>[r]\ar@<-0.5ex>[r] & \prod_{j,j'}\Hom_{\mathbf{IndSch}}(Y_j\times_Y Y_{j'}, X) 
    }
    \]
is exact. In particular, every ind-scheme is a sheaf with respect to fpqc topology on $\mathbf{Aff}$.

\end{Lemma}
\begin{proof}
    See \cite[Lemma 1.4]{Richarz}
\end{proof}

We now develop the pointed versions of Proposition~\ref{simp} and Theorem~\ref{acyclic}.

\begin{Con}\label{simp0_pointed}
    Let $\mathcal{C}$ be the category of pointed presheaves on affine schemes $Psh_*$, and let $\mathcal{D}$ be the category of pointed presheaves on affine schemes with values in commutative monoids, where the basepoint serves as the unit element. Then the reduced symmetric power functor $S^\bullet_{red}$ maps $\mathcal{C}$ to $\mathcal{D}$ and is left adjoint to the forgetful functor $U\colon \mathcal{D}\to \mathcal{C}$. Thus, the composition $\perp = S^\bullet_{red}\circ U$ forms a cotriple.
\end{Con}
\begin{Prop}\label{simp_pointed}
    Let $G$ be a commutative algebraic group over $k$. The augmented simplicial presheaf of pointed monoids on affine schemes over $k$ 
    \[
    \xymatrix{
    \ldots \ar@<1ex>[r]\ar@<0ex>[r]\ar@<-1ex>[r] &S^\bullet_{red}(S^\bullet_{red}(G))\ar@<-.5ex>[r] \ar@<.5ex>[r] & S^\bullet_{red}(G)\ar[r] & \underline{G}
    }
    \]
    induced by the cotriple $\perp = S^\bullet_{red}\circ U$ from Construction~\ref{simp0_pointed} is contractible when considered as simplicial object in $Psh_*$.
\end{Prop}

\begin{proof}
    The proof follows the same argument as that of Proposition~\ref{simp}. 
\end{proof}

\begin{Rem}
    Let $U\in Sm$ be a smooth separated scheme over $k$ and $(X, x_0)$ be any object in $\mathbf{Var_*}$. Define $S^\bullet_{red}(X)(U)$ as $\Hom_{Psh_*}(U,S^\bullet_{red}(X))$. When $U$ is affine, we have the isomorphism:
    \begin{equation}\label{transfers_str}
        S^\bullet_{red}(X)(U)\simeq \colim_n Mor_{Sch}(U,S^n(X)).
    \end{equation}
    
    Moreover, by Lemma~\ref{fpqc}, this isomorphism extends to arbitrary $U\in Sm$. Consequently, we obtain a sequence of isomorphisms:
    \[
    S^\bullet_{red}(X)(U)^+\simeq S^\bullet(X)(U)^+/ S^\bullet(x_0)(U)^+\simeq \Z_{tr}[X/x_0](U),
    \]
    where the first isomorphism follows  from~\ref{transfers_str} for arbitrary $U\in Sm$, and the second from Proposition~\ref{SV} and the definition of $\Z_{tr}[X/x_0]$. Thus, $S^\bullet_{red}(X)^+$ has a natural transfer structure and is a direct summand of $\Z_{tr}[X]$.

    For the iterated reduced symmetric power of a pointed quasi-projective $k$-variety $(X,x_0)\in\mathbf{Var_*}$, we have:
    \begin{align*}
        (S^\bullet_{red})^{\circ k}(X)(U)^+&\simeq \colim_{(n_1,\ldots,n_k)} S^\bullet_{red}(S^{n_{k-1}}\ldots S^{n_1}(X))(U)^+ \\
        &\simeq \colim_{(n_1,\ldots,n_k)}\Z_{tr}[S^{n_{k-1}}\ldots S^{n_1}(X)/x_0](U) \\
        &\simeq \colim_t \Z_{tr}[(S^t)^{\circ (k-1)}(X)/x_0](U),
    \end{align*}
    where the last isomorphism follows from Example~\ref{colimits}.
\end{Rem}

We now apply this machinery to obtain a pointed version of the resolution of $\underline{G}$ in both $Psh(cor_k)$ and $\mathcal{PS}h_{tr}$, which will serve as our primary computational tool in what follows.

\begin{Th}\label{acyclic_pointed}
Let $G$ be a commutative algebraic group over $k$. Then the augmented simplicial object in $Psh_*$ from Proposition~\ref{simp_pointed} induces a resolution of $\underline{G}$ in $PSh(cor_k)$ of the form:
    \[
    \xymatrix{
    \ldots\ar[r] & (S^\bullet_{red})^{\circ 3}(G)^+\ar[r] &(S^\bullet_{red})^{\circ 2}(G)^+\ar[r]&S^\bullet_{red}(G)^+\ar[r]&\underline{G}.
    }
    \]
This statement remains valid when replacing $Psh(cor_k)$ with $\mathcal{PS}h_{tr}$.
\end{Th}
\begin{proof}
    Following the argument in the proof of Theorem~\ref{acyclic}, we obtain an augmented aspherical simplicial abelian group
    \[
        \xymatrix{
        \ldots \ar@<1ex>[r]\ar@<0ex>[r]\ar@<-1ex>[r] &Mor_{Psh}(U, (S^\bullet_{red})^{\circ 2}(G))^+\ar@<-.5ex>[r] \ar@<.5ex>[r] &Mor_{Psh}(U, S^\bullet_{red}(G))^+\ar[r] & Mor_{Sch}(U, G)
        }
    \]
    for any $U\in Sm$. By Dold-Kan correspondence, the associated chain complex 
    \[
    \xymatrix{
    \ldots \ar[r]&Mor_{Psh}(U, (S^\bullet_{red})^{\circ 2}(G))^+\ar[r] &Mor_{Psh}(U, S^\bullet_{red}(G))^+\ar[r] & Mor_{Sch}(U, G)
    }
    \]
    is exact.
    
    The result for $\mathcal{PS}h_{tr}$ follows similarly: we take $U$ to be an arbitrary quasi-projective $k$-variety and apply the same sequence of computations.
\end{proof}

\subsection{Main results} 
\begin{Prop}\label{resolutionZ}
Let $G$ be a commutative algebraic group over $k$. For any integer $n\geq 1$, we have
\[
    \RHom_{Sh_{Nis}(cor_k)}((S^\bullet_{red})^{\circ n}(G)^+,\Z)\simeq 0.
\]
\end{Prop}
\begin{proof}
For $n=1$, we know that $S_{red}^\bullet(G)^+\simeq \Z_{tr}[G/e]$ is a direct summand of $\Z_{tr}[G]$. By \cite[Proposition 13.10 and Proposition 14.16]{MVW}, we have
\[
    \Ext^i_{Sh_{Nis}(cor_k)}(\Z_{tr}[G],\Z)\simeq H^i_{Nis}(G, \Z)\simeq H^i_{Zar}(G,\Z)=0
\]
for all $i\geq 1$. This yields the following sequence of isomorphisms:
\begin{align*}
    \RHom_{Sh_{Nis}(cor_k)}(S_{red}^\bullet(G)^+,\Z)&\simeq\Hom_{Sh_{Nis}(cor_k)}(S_{red}^\bullet(G)^+,\Z)\\
    &\simeq \Z_{tr}[\Spec(k)](G)/ \Z_{tr}[\Spec(k)](e)\\
    &\simeq 0.
\end{align*}
For arbitrary $n\geq 1$, we obtain:
\[
    (S^\bullet_{red})^{\circ n}(G)^+\simeq\colim_{(k_1,\ldots, k_{n-1})}S^\bullet_{red}(S^{k_{n-1}}\ldots S^{k_1}G)\simeq\colim_{(k_1,\ldots, k_{n-1})}\Z_{tr}[S^{k_{n-1}}\ldots S^{k_1}G/e].
\]
Since the functor $i\mapsto (i,\ldots, i)$ from $\mathcal{N}$ to $\mathcal{N}\times\ldots \times \mathcal{N}$ is final, we can rewrite the last colimit as $\colim_{t}\Z_{tr}[(S^t)^{\circ(n-1)}(G)/e]$.
Therefore:
\begin{align*}
    \RHom_{Sh_{Nis}(cor_k)}((S^\bullet_{red})^{\circ n})(G)^+,\Z) &\simeq \RHom_{Sh_{Nis}(cor_k)}(\colim_{t}\Z_{tr}[(S^t)^{\circ(n-1)}(G)/e],\Z) \\
    &\simeq R\lim_{t}\RHom_{Sh_{Nis}(cor_k)}(\Z_{tr}[(S^t)^{\circ(n-1)}(G)/e],\Z). 
\end{align*}
For all $i\geq 1$, we have the following sequence of isomorphisms
\[
\Ext^i_{Sh_{Nis}(cor_k)}(\Z_{tr}[(S^t)^{\circ(n-1)}(G)], \Z)\simeq H_{cdh}^i((S^t)^{\circ(n-1)}(G),\Z)=0,
\]
where the first isomorphism follows from  \cite[Theorem 14.20]{MVW} and the second follows from  Corollary~\ref{quotient} and Theorem~\ref{Comparison}. Since $\Z_{tr}[(S^t)^{\circ(n-1)}(G)/e]$ is a direct summand of $\Z_{tr}[(S^t)^{\circ(n-1)}(G)]$, we conclude that
\begin{align*}
    \RHom_{Sh_{Nis}(cor_k)}(\Z_{tr}[(S^t)^{\circ(n-1)}(G)/e],\Z)&\simeq \Hom_{Sh_{Nis}(cor_k)}(\Z_{tr}[(S^t)^{\circ(n-1)}(G)/e],\Z)\\&\simeq 0.
\end{align*}
Thus, $\RHom_{Sh_{Nis}(cor_k)}((S^\bullet_{red})^{\circ n})(G)^+,\Z)$ is trivial. 
\end{proof}
\begin{Th}\label{ext}
    Let $G$ be a commutative algebraic group over $k$. Then
    \[\RHom_{Sh_{Nis}(cor_k)}(\underline{G}, \Z)\simeq 0.
    \]
\end{Th}
\begin{proof} By Proposition~\ref{resolutionZ}, we can compute the required $\RHom$ in $Sh_{Nis}(cor_k)$ using the resolution from Theorem~\ref{acyclic_pointed}:
 \[
    \xymatrix{
    \ldots\ar[r] & (S^\bullet_{red})^{\circ 3}(G)^+\ar[r] &  (S^\bullet_{red})^{\circ 2}(G)^+ \ar[r]&S^\bullet_{red}(G)^+\ar[r]&\underline{G}.
    }
\]
After applying $\Hom_{Sh_{Nis}(cor_k)}(-, \Z)$ we get the zero complex and the claim follows.
\end{proof}
\begin{Cor}
    Let $G$ be a semi-abelian variety over $k$. Then
    \[
    \RHom_{\DMeN(k,\Z)}(M_1(G), \Z)\simeq 0.
    \]
\end{Cor}
\begin{proof}
This follows immediately from Corollary~\ref{ext}.
\end{proof}

\begin{Prop}\label{resolutionAB}
    Let $A$ and $B$ be abelian varieties over $k$. Then for any $n\geq 1$, the complex $\RHom_{\mathcal{PS}h_{tr}}((S^\bullet_{red})^{\circ n}(A)^+,\underline{B})$ is concentrated in degree $0$ and is isomorphic to both $Mor_{Psh_*}((S^\bullet_{red})^{\circ (n-1)}(A),\underline{B})$ and $Mor_{Sch_*}(A,B)$.
\end{Prop}
\begin{proof}
For $n=1$, since $S^\bullet_{red}(A)^+\simeq \Z_{tr}[A/e]$ is projective as a direct summand of $\Z_{tr}[A]$, we have: 
\begin{align*}
        \RHom_{\mathcal{PS}h_{tr}}((S^\bullet_{red})(A)^+,\underline{B}) &\simeq \Hom_{\mathcal{PS}h_{tr}}((S^\bullet_{red})(A)^+,\underline{B}) \\
        &\simeq B(A)/B(e) \\
        &\simeq Mor_{Sch_*}(A, B).
\end{align*}
For $n=2$,  $(S^\bullet_{red})^{\circ 2}(A)^+\simeq\colim_n \Z_{tr}[S^nA/e]$. This yields the following sequence of isomorphisms:
\begin{align*}
        \RHom_{\mathcal{PS}h_{tr}}((S^\bullet_{red})^{\circ 2}(A)^+,\underline{B}) &\simeq \RHom_{\mathcal{PS}h_{tr}}(\colim_k \Z_{tr}[S^kA/e],\underline{B}) \\
        &\simeq R\lim_k\RHom_{\mathcal{PS}h_{tr}}(\Z_{tr}[S^kA/e],\underline{B}) \\
        &\simeq R\lim_k\Hom_{\mathcal{PS}h_{tr}}(\Z_{tr}[S^kA/e],\underline{B})\\
        &\simeq R\lim_k Mor_{Sch_*}(S^kA, B).
\end{align*}
All isomorphisms above are straightforward except the third, which follows from the projectivity of $\Z_{tr}[S^kA/e]$. For all $k \geq 1$, the transition maps in the last inductive system induce isomorphisms $Mor_{Sch_*}(S^{k+1}A,B)\to Mor_{Sch_*}(S^kA,B)$. To verify this, take $k=2$, then surjectivity is immediate, and for injectivity, suppose two maps have identical images in $Mor_{Sch_*}(A,B)$. Their difference then induces a map $d\colon A\times A\to B$ with $d(e_A,A)=d(A,e_A)=e_B$. By the rigidity theorem, $d$ must be the identity map. This argument extends easily to show that the transition maps $Mor_{Sch_*}(S^{k+1}A,B)\to Mor_{Sch_*}(S^kA,B)$ are isomorphisms for all $k\geq 1$. Since the final map $Mor_{Sch_*}(A, B)\to Mor_{Sch_*}(e, B)$ is surjective, this inverse system satisfies the Mittag-Leffler condition. Therefore, 
\[
    R\lim_k Mor_{Sch_*}(S^kA, B) \simeq \lim_k Mor_{Sch_*}(S^kA, B) \simeq Mor_{Sch_*}(A,B)
\]
and the result follows for $n=2$.

For $n\geq 3$, applying the same argument as in the proof of Proposition~\ref{resolutionZ} and  using the projectivity of $\Z_{tr}[(S^t)^{\circ(n-1)}(A)/e]$, we obtain:  \begin{align*}
    \RHom_{\mathcal{PS}h_{tr}}((S^\bullet_{red})^{\circ n}(A)^+, \underline{B})&\simeq R\lim_t\RHom_{\mathcal{PS}h_{tr}}(\Z_{tr}[(S^t)^{\circ(n-1)}(A)/e],\underline{B}) \\
    &\simeq R\lim_t\Hom_{\mathcal{PS}h_{tr}}(\Z_{tr}[(S^t)^{\circ(n-1)}(A)/e],\underline{B}) \\
    &\simeq R\lim_t Mor_{Sch_*}((S^t)^{\circ (n-1)}(A),B)
\end{align*}
Iterating the argument from the $n=2$ case shows that the transition maps
\[
Mor_{Sch_*}((S^{(t+1)})^{\circ (n-1)}(A),B)\to Mor_{Sch_*}((S^t)^{\circ (n-1)}(A),B)
\]
are isomorphisms for all $t\geq 1$. Thus, this inverse system satisfies the Mittag-Leffler condition, yielding:
\[
    R\lim_t Mor_{Sch_*}((S^t)^{\circ (n-1)}(A), B) \simeq \lim_t Mor_{Sch_*}((S^t)^{\circ (n-1)}(A), B) \simeq Mor_{Sch_*}(A,B).
\]
Furthermore, 
\begin{align*}
    \lim_t Mor_{Sch_*}((S^t)^{\circ (n-1)}(A), B) &\simeq Mor_{Psh_*}(\colim_t(S^t)^{\circ (n-1)}(A), \underline{B}) \\
    &\simeq Mor_{Psh_*}(\colim_{(k_1,\ldots,k_{n-1})}S^{k_{n-1}}\ldots S^{k_1}A,\underline{B})\\
    &\simeq Mor_{Psh_*}((S^\bullet_{red})^{\circ(n-1)}(A),\underline{B}),
\end{align*}
where the second isomorphism holds because the functor $i\mapsto (i,\ldots, i)$ from $\mathcal{N}^{op}$ to $\mathcal{N}^{op}\times\ldots\times \mathcal{N}^{op}$ is initial.
\end{proof}

\begin{Th}\label{ABV}
Let $A$, $B$ be abelian varieties over $k$. Then $\RHom_{\mathcal{PS}h_{tr}}(\underline{A},\underline{B})$ is concentrated in degree $0$ and is quasi-isomorphic to $\Hom_{\mathbf{Ab_k}}(A,B)$. In particular,
\[
    \Ext^i_{\mathcal{PS}h_{tr}}(\underline{A}, \underline{B})\simeq 0, \text{ for } i\geq 1.
\]
\end{Th}

\begin{proof}
By Proposition~\ref{resolutionAB}, we can compute $\RHom_{\mathcal{PS}h_{tr}}(\underline{A}, \underline{B})$ using the (simplicial) resolution of $\underline{A}$ from Theorem~\ref{acyclic_pointed}. After applying $\Hom_{\mathcal{PS}h_{tr}}(-, \underline{B})$ and using the isomorphism $\Hom_{\mathcal{PS}h_{tr}}((S^\bullet_{red})^{\circ k}(A)^+, \underline{B})\simeq Mor_{Psh_*}((S^\bullet_{red})^{\circ (k-1)}(A), \underline{B})$ (Proposition~\ref{resolutionAB}), we get a cosimplicial object $C^*$:
\[
\xymatrix{
Mor_{Psh_*}(\underline{A}, \underline{B})\ar@<0.5ex>[r]\ar@<-0.5ex>[r]  & Mor_{Psh_*}(S^\bullet_{red}(A), \underline{B})\ar@<1ex>[r]\ar@<0ex>[r]\ar@<-1ex>[r] & Mor_{Psh_*}((S^\bullet_{red})^{\circ 2}(A), \underline{B})\ar@<1ex>[r]\ar@<0.33ex>[r]\ar@<-1ex>[r]]\ar@<-0.33ex>[r] &\ldots
}.
\]
By Proposition~\ref{resolutionAB}, all terms of $C^*$ are isomorphic to $Mor_{Sch_*}(A,B)$. We split the computation of cohomology of its associated complex $C^\bullet$ in several steps.

First, we show that the maps
\[\xymatrix{
Mor_{Psh_*}(\underline{A}, \underline{B})\ar@<0.5ex>[r]\ar@<-0.5ex>[r]  & Mor_{Psh_*}(S^\bullet_{red}(A), \underline{B})}
\]
in degree $0$ coincide. Indeed, for any $n\geq 1$, at the finite level, one map sends $f\colon A\to B$ to the composition $S^n(A)\to S^n(B)\to B$, inducing a map $A^n\to B$ that  sends $(a_1, \ldots, a_n)$ to the sum $\sum_{i=1}^n f(a_i)$. 

The second map sends $f\colon A\to B$ to the composition of $S^n(A)\to A\to B$, where the first map is induced by multiplication in $A$. This gives a map $A^n\to B$ sending $(a_1,\ldots, a_n)$ to $f(\sum_{i=1}^{n}a_i)$. Since $f$ preserves the identity element, the usual corollary of the rigidity of the abelian variety $B$ implies that this map equals $\sum_{i=1}^{n}f(a_i)$. Taking the colimit, we conclude that the maps in degree zero are equal.

Furthermore, we show that these maps become the identity after the identification
\[
i^*\colon Mor_{Psh_*}(S_{red}^\bullet(A), \underline{B})\to Mor_{Sch_*}(A, B),
\]
where $i^*$ is induced by $i\colon\underline{A}\to S^\bullet_{red}(A)$. To check this, it suffices to verify that  $f\circ m_A\circ i=f$, which follows immediately from $m_A \circ i=id$. 

We observe that in each degree of $C^*$, all but the one of the maps are induced by the corresponding maps in the simplicial presheaf of Proposition~\ref{simp_pointed}. The additional arrow arises from the functionality of the reduced symmetric power functor and the map $S^\bullet_{red}(B)\to \underline{B}$ induced by multiplication in $B$. We also recall that the identification \[Mor_{Psh_*}((S_{red}^\bullet)^{\circ m}(A), \underline{B})\to Mor_{Psh_*}(\underline{A},\underline{B})\simeq Mor_{Sch_*}(A,B)
\]
is induced by a sequence of natural inclusions:
\[
\underline{A}\to S_{red}^\bullet(A)\to S_{red}^\bullet(S_{red}^\bullet(A))\to\ldots\to (S_{red}^\bullet)^{\circ m}(A).
\]

We now extend our earlier argument to show that all maps in $C^*$ are identity isomorphisms. Let $f\colon S^\bullet_{red}(A)^{\circ(m)}\to \underline{B}$ and consider 
\[
\partial_j^*\colon Mor_{Psh_*}((S^\bullet_{red})^{\circ(m)}(A)), \underline{B})\to Mor_{Psh_*}((S^\bullet_{red})^{\circ(m+1)}(A), \underline{B})
\]
induced by a face map $\partial_j\colon (S^\bullet_{red})^{\circ(m+1)}(A)\to (S^\bullet_{red})^{\circ m}(A)$. Then  $f$ corresponds to the map 
\[
\xymatrix{\underline{A}\ar[r]^-{i_0} & S_{red}^\bullet(A)\ar[r]^-{i_1}&\ldots \ar[r]^-{i_{m-2}} &  (S^\bullet_{red})^{\circ(m-1)}(A)\ar[r]^-{i_{m-1}}&  (S^\bullet_{red})^{\circ m}(A)\ar[r]^-f &\underline{B}
}
\]
in $Mor_{Psh_*}(\underline{A}, \underline{B})$. We must verify that $f\circ i^{\circ m}:\underline{A}\to (S^\bullet_{red})^{\circ m}(A)\to \underline{B}$ coincides with the composition $ f\circ \partial_j\circ i^{\circ(m+1)}\colon \underline{A}\to (S^\bullet_{red})^{\circ(m+1)}(A)\to \underline{B}$. For this, we prove the identity $i^{\circ m}=\partial_j\circ i^{\circ (m+1)}$ by induction on $m$. The base case is clear since $\partial_0\circ i_0=m_A\circ i_0\colon \underline{A}\to S_{red}^\bullet(A)\to \underline{A}$ is the identity. Since  $i$ is exactly the contraction in the simplicial presheaf of Proposition~\ref{simp_pointed}, we have $\partial_0\circ i_{m+1}=id$ and $\partial_j\circ i_{m+1}=i_m\circ \partial_{j-1}$. Using these relations  and the induction hypothesis, we see that $\partial_0 \circ i^{\circ(m+1)}= i^{\circ m }$ and $\partial_j\circ i^{\circ(m+1)}= i_m\circ\partial_{j-1}\circ i^{\circ m}=i_m\circ i^{\circ(m-1)}=i^{\circ m}$. 

The remaining map 
\[Mor_{Psh_*}((S^\bullet_{red})^{\circ m}(A), \underline{B})\to Mor_{Psh_*}((S^\bullet_{red})^{\circ(m+1)}(A), \underline{B})\]
sends $f\colon (S^\bullet_{red})^{\circ m}(A) \to \underline{B}$ to $m_B\circ s(f)\colon (S^\bullet_{red})^{\circ(m+1)}(A)\to S^\bullet_{red}(B)\to \underline{B}$. As in the degree $0$ case, this map is the identity after identification of both terms with $Mor_{Psh_*}(\underline{A}, \underline{B})$, since the diagram 
\[
\xymatrix{&(S^\bullet_{red})^{\circ (m+1)}(A) \ar[r]^-{s(f)}\ar@<+0.5ex>[d]^-m & S^\bullet_{red}(B)\ar[d]^m\\
\underline{A}\ar[r]^-{i^{\circ(m)}}&(S^\bullet_{red})^{\circ m}(A)\ar[r]^-f\ar@<0.5ex>[u]^-i & \underline{B}}
\]
is commutative.

We conclude that all terms of $C^*$ are isomorphic to $Mor_{Sch_*}(A, B)$ and all maps in each degree are identities. Therefore, the associated chain complex $C^\bullet$ takes the form: 
\[
\xymatrix{
Mor_{Sch_k, *}(A, B)\ar[r]^-0 & Mor_{Sch_k, *}(A, B)\ar[r]^-{id}& Mor_{Sch_k, *}(A, B)\ar[r]^-0 & \ldots.}
\]
This complex has only zeroth cohomology, which is isomorphic to $Mor_{Sch_*}(A, B)$. By the rigidity theorem, this is precisely $\Hom_{\mathbf{Ab_k}}(A, B)$, completing the proof.
\end{proof}

\begin{Cor}
Let $A$, $B$ be abelian varieties over $k$. Then $\RHom_{Sh_{Nis}(Cor_k)}(\underline{A}, \underline{B})$ is concentrated in degree $0$ and is quasi-isomorphic to $\Hom_{\mathbf{Ab_k}}(A, B)$.
\end{Cor}
\begin{proof}
We can follow the lines of the proof of Theorem~\ref{ABV} if we show that for all $n\geq 1$, $(S^\bullet_{red})^{\circ n}(A)^+$ is $\Hom_{Sh_{Nis}(Cor_k)}(-,\underline B)$-acyclic. For this, we can use the same argument as in the proof of Proposition~\ref{resolutionAB}, provided we prove that for all $i\geq 1$, $\Ext_{Sh_{Nis}(Cor_k)}^i(\Z_{tr}[(S^t)^{\circ (n-1)}(A)/e],\underline{B})$ vanishes for all $t\geq 1$ and all $n\geq1$. To establish this, it suffices to show that for any smooth $k$-variety $X$ with an action of a finite group $G$ and for all $i\geq 1$, $\Ext_{Sh_{Nis}(Cor_k)}^i(\Z_{tr}[(X/G)/e],\underline{B})$ is zero. Let us denote by $Y$ the quotient $X/G$. Since $\Z_{tr}[Y/e]$ is a direct summand of $\Z_{tr}[Y]$, it is enough to prove the vanishing of $\Ext^i_{Sh_{Nis}(Cor_k)}(\Z_{tr}[Y],\underline{B})$. As this is isomorphic to $H^i_{Nis}(Y,\underline{B})$, we demonstrate the vanishing of these cohomology groups for all $i\geq 1$. 

First, recall that any rational map $X\dashrightarrow B$ from a smooth $k$-variety $X$ extends to a morphism $X\to B$ by \cite[Theorem 1.18]{dGP}. Given an open subset $U\hookrightarrow Y$ with a map $U\to B$, consider the pullback diagram:
\[
\xymatrix{
    U'\ar[r]\ar[d] & X\ar[d]\\
    U\ar[r] & Y.
}
\]
By the previous observation, we can extend a map $U'\to B$ to $X\to B$. This map is $G$-invariant since it has this property on the dense subset $U'$. Therefore, it descends to $Y$, implying that any map $U\to B$ can be extended to $Y$.

To prove that the natural map $\underline{B}(Y)\to R\Gamma_{Nis}\underline{B}(Y)$ is an isomorphism, we can verify the Mayer-Vietoris property for the small Nisnevich site $Y_{Nis}$ (for a discussion of the Mayer-Vietoris property, see \cite[Section 3]{CHSW}, or for a simpler proof in the case of the Nisnevich topology, see \cite[Proposition 4.2.1]{Beil_Vologod}. While this proof deals with the big Nisnevich site, the argument also applies to the small site).

Take any \'etale map $s\colon U'\to Y$ and consider the elementary Nisnevich square
\[
\xymatrix{
W\ar[r]^j\ar[d]^w & V\ar[d]^p \\
U\ar[r]^i & U',
}
\]
where $i$ is an open embedding and $p$ is \'etale. To verify the Mayer-Vietoris property, it suffices to show that the map $\underline{B}(V)\oplus\underline{B}(U)\to \underline{B}(W)$ is surjective. To do this, we consider the pullback diagram
\[
\xymatrix{
    & V\times_Y X\ar[r]\ar[d]^{pr_1} & X\ar[d]^{pr} \\
    W\ar[r]^j & V\ar[r]^{s \circ p} & Y.
}
\]
Since the map $pr$ is the universal categorical quotient and $s\circ p$ is \'etale, $pr_1$ induces an isomorphism $(V\times_Y X)/G\simeq V$, identifying $V$ with a quotient of a smooth $k$-variety by a finite group $G$. Since $j\colon W\to V$ is open, it follows from our earlier argument that $\underline{B}(V)$ surjects onto $\underline{B}(W)$. Hence, $\underline{B}$ satisfies the Mayer-Vietoris property on $Y_{Nis}$, implying that $H^i_{Nis}(Y,\underline{B})$ vanishes for all $i\geq 1$, which completes the proof. 
\end{proof}


\begin{thebibliography}{10}

\bibitem{AEWH}
G. Ancona, S. Enright-Ward, and A. Huber, \emph{On the motive of a commutative algebraic group}, Documenta Math. \textbf{20} (2015), 807--858.


\bibitem{Beil_Drin}
A. Beilinson and V. Drinfeld, \emph{Quantization of Hitchin's integrable system and Hecke eigensheaves}, Available at \url{http://www.math.uchicago.edu/~drinfeld/langlands/QuantizationHitchin.pdf}.


\bibitem{Beil_Vologod}
A. Beilinson and V. Vologodsky, \emph{A DG Guide to Voevodsky's Motives}, GAFA Geom. funct. anal. \textbf{17} (2008), 1709--1787.


\bibitem{BerI}
V. Berkovich, \emph{Spectral Theory and Analytic Geometry over Non-Archimedean Fields},
Surveys and Monographs 33, Amer. Math. Soc., Providence, 1990.


\bibitem{BerII}
V. Berkovich, \emph{\'Etale cohomology for non-Archimedean analytic spaces}, Publ. Math. IH\'ES \textbf{78} (1993), 5--161.


\bibitem{BerIII}
V. Berkovich, \emph{Smooth $p$-adic analytic spaces are locally contractible}, Invent. math. \textbf{137} (1999), 1--84.


\bibitem{CHSW}
G.~Corti{\~n}as, C.~Haesemeyer, M.~Schlichting, and C.~Weibel, \emph{Cyclic homology, cdh-cohomology and negative K-theory}, Ann. Math. \textbf{167} (2008), no.~2, 549--573.


\bibitem{ERW}
J. Ebert and O. Randal-Williams, \emph{Semisimplicial spaces}, Algebraic \& Geometric Topology \textbf{19} (2019), 2099--2150.


\bibitem{Iverson}
B. Iversen, \emph{Cohomology of Sheaves}, Universitext, 
Springer-Verlag, Berlin Heidelberg, 1986.

\bibitem{KashShap} M. Kashiwara and P. Schapira, \emph{Categories and
sheaves}, Grund\-lehren der Mathematischen Wissenschaften, {\bf 332},
Springer-Verlag, Berlin, 2006.

\bibitem{MVW}
C. Mazza, V. Voevodsky, and C. Weibel, \emph{Lecture notes on motivic cohomology}, Clay Mathematics Monographs, 2. American Mathematical Society, Providence, RI; Clay Mathematics Institute, Cambridge, MA, 2006.

\bibitem{Org}
F. Orgogozo, \emph{Isomotifs de dimension inf\'erieure ou \'egale \`a un}, Manuscripta Math. \textbf{115} (2004), 339--360.


\bibitem{Richarz}
T. Richarz, \emph{Basics on affine Grassmannians}, Available at \url{https://www.mathematik.tu-darmstadt.de/media/algebra/homepages/richarz/Notes_on_affine_Grassmannians.pdf}, 2020.


\bibitem{SSV}
S. Sagave, T. Sch\"urg, and G. Vezzosi, \emph{Derived logarithmic geometry {I}},  Journal of the Institute of Mathematics of Jussieu \textbf{15} (2016), no.~2, 367--405.

\bibitem{SS}
M. Spie{\ss} and T. Szamuely, \emph{On the Albanese map for smooth quasi-projective varieties}, Math. Ann. \textbf{325} (2003), 1--17.

\bibitem{SV} A. Suslin and V. Voevodsky, \emph{Bloch-Kato conjecture and motivic cohomology with finite coefficients}, in The arithmetic and geometry of algebraic cycles (Banff, AB, 1998), 117--189, NATO Sci. Ser. C Math. Phys. Sci., \textbf{548}, Kluwer, 2000.

\bibitem{Sing}
A. Suslin and V. Voevodsky, \emph{Singular homology of abstract algebraic varieties}, Invent. Math. \textbf{123} (1996), 61--94.


\bibitem{THU}
A. Thuillier, \emph{G\'eom\'etrie toro\"{i}dale et g\'eom\'etrie analytique non archim\'edienne. Application au type d'homotopie de certains sch\'emas formels}, Manuscripta Math. \textbf{123} (2007), 381--451.



\bibitem{dGP}
G. van der Geer and B. Moonen, \emph{Abelian varieties}, Book under preparation, Available at \url{http://van-der-geer.nl/~gerard/AV.pdf}.


\bibitem{VoevodI}
V. Voevodsky, \emph{Homotopy theory of simplicial sheaves in completely decomposable topologies}, Journal of Pure and Applied Algebra \textbf{214} (2010), no.~8, 1384--1398.


\bibitem{VoevodII}
V. Voevodsky, \emph{Unstable motivic homotopy categories in Nisnevich and cdh-topologies}, Journal of Pure and Applied Algebra \textbf{214} (2010), no.~8, 1399--1406.

\bibitem{Weibel}
C. Weibel, \emph{An Introduction to Homological Algebra},
Cambridge Studies in Advanced Mathematics, Cambridge University Press, 1994.

\bibitem{Zhu}
X. Zhu, \emph{An introduction to affine Grassmannians and the geometric Satake equivalence}, IAS/Park City Mathematics Series {\bf 24} (2017), 59--154.

\end{thebibliography}
\end{document}